\documentclass[reqno]{amsart}
\usepackage[margin=1.5in]{geometry}
\usepackage{amsmath, amsfonts, amssymb, amsthm,  cases}
\usepackage{amsaddr} 
\usepackage{graphicx}
\usepackage{mathrsfs}
\usepackage{color}
\usepackage{verbatim}
\usepackage{float}
\theoremstyle{definition}
\usepackage{setspace}

\DeclareMathOperator{\RR}{{\mathbb{R}}}
\DeclareMathOperator{\BB}{{\mathbb{B}}}
\DeclareMathOperator{\Ss}{{\mathbb{S}}}
\DeclareMathOperator{\that}{{\hat{\theta}}}

\DeclareMathOperator{\UT}{{\widetilde{U}}}
\DeclareMathOperator{\AF}{{\mathcal{A}}}
\DeclareMathOperator{\HH}{{\mathcal{H}}}
\DeclareMathOperator{\BigO}{{O}}

\newcommand{\Mm}{{M_{-}}}
\newcommand{\Mp}{{M_{+}}}
\newcommand{\Mmd}{{M_{-}'}}
\newcommand{\Mpd}{{M_{+}'}}

\newcommand{\vn}{{\nu}}

\newcommand{\upsh}{{u_{+}^{\#}}}
\newcommand{\umsh}{{u_{-}^{\#}}}

\newcommand{\domain}{{H^2_0(\Omega,\alpha)}}

\definecolor{plum}{rgb}{1.0, 0.0, 1.0}

\newtheorem{thm}{Theorem}

\newtheorem{prop}[thm]{Proposition}

\newtheorem{remark}[thm]{Remark}

\title{On Clamped Plates with Log-Convex Density}
\author{L. M. Chasman}
\address{University of Minnesota - Morris, 600 E. 4th Street, Morris, MN 56267, USA} 
\email{chasmanm$@$morris.umn.edu}
\author{Jeffrey J. Langford}
\address{Bucknell University, 1 Dent Drive, Lewisburg, PA 17837, USA}
\email{Jeffrey.Langford$@$bucknell.edu}
\date{\today}


\begin{document}
\begin{abstract} We consider the analogue of Rayleigh's conjecture for the clamped plate in Euclidean space weighted by a log-convex density. We show that the lowest eigenvalue of the bi-Laplace operator with drift in a given domain is bounded below by a constant $C(V,n)$ times the lowest eigenvalue of a centered ball of the same volume; the constant depends on the volume $V$ of the domain and the dimension $n$ of the ambient space. Our result is driven by a comparison theorem in the spirit of Talenti, and the constant $C(V,n)$ is defined in terms of a minimization problem following the work of Ashbaugh and Benguria. When the density is an ``anti-Gaussian," we estimate $C(V,n)$ using a delicate analysis that involves confluent hypergeometric functions, and we illustrate numerically that $C(V,n)$ is close to $1$ for low dimensions.
\end{abstract}
\maketitle

\section{Introduction}
\subsection*{History}

Given a bounded domain $\Omega \subset \mathbb{R}^n$, the frequencies and modes of vibration for a clamped plate of shape $\Omega$ are governed by the eigenvalue problem
\begin{equation}\label{Eqn:CPP}
  \begin{cases}\Delta^2 u=\Lambda u &\text{in $\Omega$},\\
  u=0&\text{on $\partial\Omega$},\\
  \frac{\partial u}{\partial \vn}=0&\text{on $\partial\Omega$},
  \end{cases}
\end{equation}
where $\Delta^2u=\Delta(\Delta u)$ is the bi-Laplace operator and $\frac{\partial u}{\partial \vn}$ is the exterior normal derivative.
The technique of estimating the frequencies of vibration of a clamped plate in terms of the plate's geometry originated with Lord Rayleigh, who conjectured that the lowest frequency of vibration for a clamped plate is bounded below by the corresponding frequency of a clamped disk of the same area \cite{RToS}. The first major progress towards Rayleigh's conjecture came from Szeg\H o \cite{S50} for dimension $n=2$ and from Talenti \cite{T81} for general dimensions. Nadirashvili \cite{N95} provided the first proof of Rayleigh's conjecture in dimension $n=2$, and  in the same year, Ashbaugh and Benguria \cite{AB95} independently proved the conjecture in dimensions $n=2$ and $3$. 
Although the conjecture is still open in dimensions $n\geq 4$, Ashbaugh and Laugesen \cite{AL96} gave a partial result which showed that the conjecture is ``asymptotically true in high dimensions.''

In this work, we consider the analogue of problem \eqref{Eqn:CPP} when Euclidean space is weighted by a radial log-convex function. We were led to study this problem because recently, Chambers \cite{C15} (see also the works \cite{BCM12, BCM14, BMP13, FM12, RCBM08}) proved the log-convex conjecture: 
centered balls are perimeter-minimizing when both volume and perimeter are weighted by a radial log-convex function.  Using this isoperimetric inequality, we establish a comparison result in the spirit of Talenti \cite{T81} and adapt the methodology of Ashbaugh and Benguria \cite{AB95} to study this weighted version of problem \eqref{Eqn:CPP}.

\subsection*{Main result and related literature}To precisely state our main result, we require some notation. Let $\Omega \subset \mathbb{R}^n$ denote a bounded $C^{\infty}$ domain and let $\phi:\mathbb{R} \to \mathbb{R}$ be a function that is even, convex, and $C^{\infty}$ with $\phi(|x|)$ real analytic on $\mathbb{R}^n$. We write $a(x)=e^{\phi(|x|)}$ and let $\alpha$ denote the absolutely continuous measure on $\mathbb{R}^n$ with density $a(x)$. We write
\begin{equation}\label{eqn:defA}
 \AF u=\frac{1}{a}\nabla\cdot(a\nabla u)=\Delta u+\nabla \phi \cdot \nabla u
\end{equation}
for a Laplacian with drift; this operator can also be considered as the weighted Laplacian for our weighted space. We are led to consider the eigenvalue problem
\begin{equation}\label{eqn:main pde problemma}
 \begin{cases}\AF^2u=\Lambda u &\text{in $\Omega$},\\
  u=0&\text{on $\partial\Omega$},\\
  \frac{\partial u}{\partial \vn}=0&\text{on $\partial\Omega$}.
  \end{cases} 
\end{equation}
In our main result, we obtain a lower bound on the first eigenvalue of problem \eqref{eqn:main pde problemma} in terms of the first eigenvalue of a centered ball (i.e., a ball centered at the origin) of the same measure. Namely, we prove:
\begin{thm}\label{mainthm} 
Let $\Omega$ be as above and suppose $\Lambda_1$ denotes the lowest eigenvalue of problem \eqref{eqn:main pde problemma}. Suppose $\Omega^{\#}$ denotes a centered ball having the same $\alpha-$measure as $\Omega$, and assume that either the dimension $n\geq 3$, or $n=2$ and the ground state for $\Omega^{\#}$ is radial. Then
\[
\Lambda_1(\Omega)\geq C\Lambda_1(\Omega^{\#}),
\]
for some constant $0<C=C(\alpha(\Omega),n)\leq 1$ depending on the $\alpha-$measure of $\Omega$ and the dimension $n$ of the ambient space.
\end{thm}
The additional requirement in the case of $n=2$ is due to Proposition~\ref{Thm:RadMod}, where we are only able to prove the ground state is radial for dimensions $n\geq3$. It seems quite likely that the ground state is radial for $n=2$ as well, which we discuss in a remark following the statement of the proposition.

The constant $C$ is defined in terms of a $J_{A,B}$ minimization problem following the work of Ashbaugh and Benguria \cite{AB95}. After the proof of Theorem~\ref{mainthm}, we restrict our attention to the specific weight $a(x)=e^{|x|^2/2}$, solve the associated $J_{A,B}$ minimization problem, and compute eigenvalues of balls using confluent hypergeometric functions. This analysis allows us to numerically determine explicit values of $C(V,n)$. Since volume is a function of radius $R$, we may also express the constants $C(R,n)$ as functions of the radius. These constants are plotted below for dimensions $n=2,3,4$, and $5$. Observe that in these dimensions we have $C(R,n)\geq .85$, with the minimum increasing with $n$.

\begin{figure}[h!]\label{Clowerbound}
\includegraphics[width=4in]{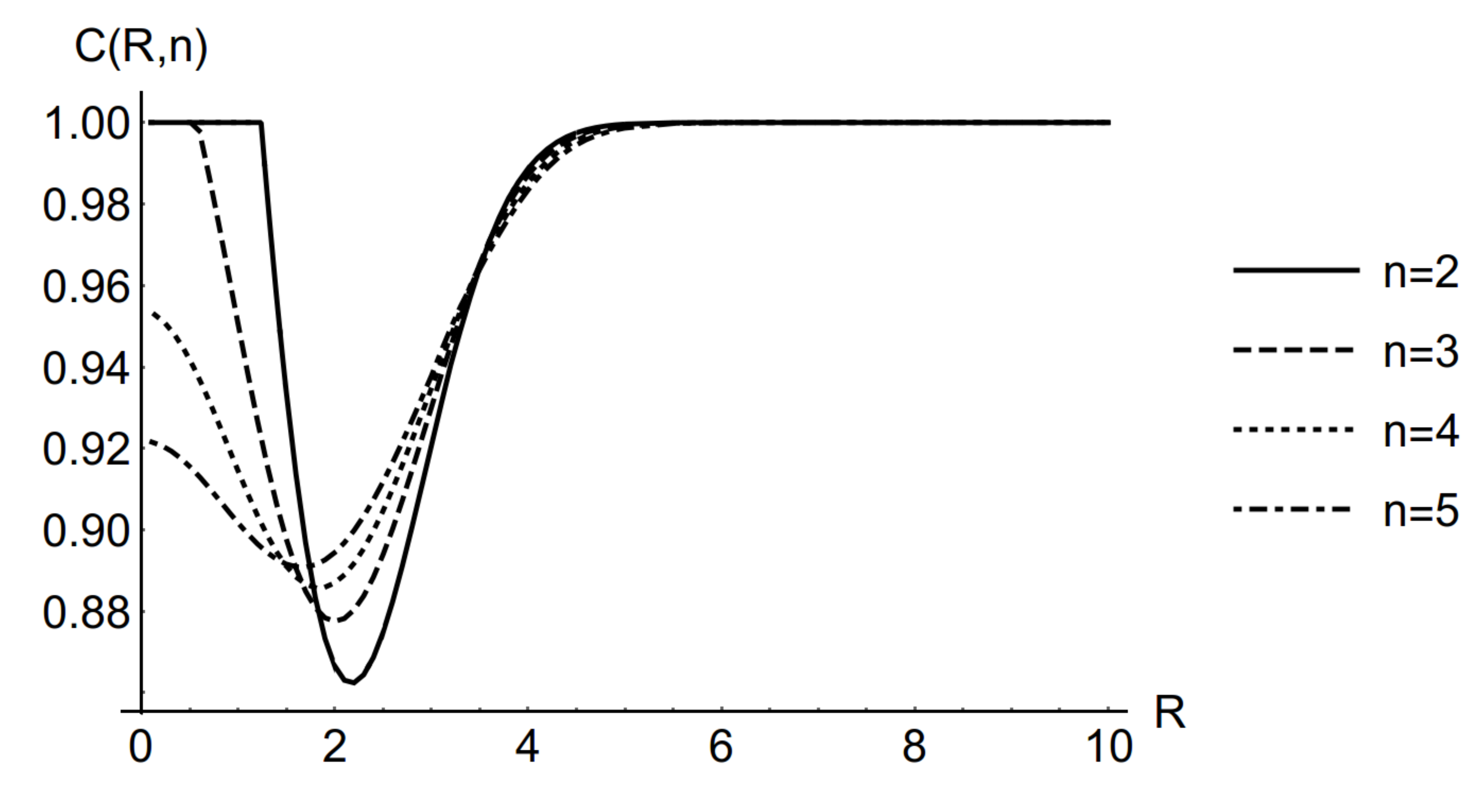}
 \caption{\label{fig:const}Graphs of the $C(R,n)$ values for low dimensions.}
\end{figure}

This paper is part of recent efforts to develop classical eigenvalue inequalities in the setting of anti-Gauss space. In \cite{BCHT15}, Brandolini, Chiacchio, Henrot, and Trombetti proved a Faber-Krahn inequality in anti-Gauss space for the analogous fixed membrane problem. Ma and Liu \cite{ML08,ML09} also studied the fixed membrane problem, and obtained lower bounds on the spectral gap for convex domains. In \cite{BCdB16}, Brock, Chiacchio, and di Blasio established a Kornhauser-Stakgold inequality in anti-Gauss space for the analogous free membrane problem. In the context of compact Riemannian manifolds with a general exponential density, Xia and Xu \cite{XX14} obtained a lower bound on the lowest clamped plate eigenvalue in terms of the lowest fixed membrane eigenvalue under certain assumptions on the manifold's Balkry-Ricci curvature. In the same Riemannian setting, Du, Wu, Li, and Xia \cite{DWLX15} obtained universal inequalities for gaps of clamped plate eigenvalues, and Ma and Li \cite{MD10} established lower bounds on eigenvalues of the drift Laplacian in terms of lower Ricci curvature bounds.

\subsection*{Outline of paper.} The rest of our paper is divided into two halves as follows. In the first half of the paper, we consider  general weighted spaces with exponential density. We identify the form of the drift Laplacian and define the relevant weighted Sobolev spaces (Section \ref{Section:Notation}). We prove the existence of eigenvalues for problem \eqref{eqn:main pde problemma} and establish regularity of the eigenfunctions (Section \ref{coercsection}). For spaces with log-convex density, we develop a symmetrization result (Section \ref{Section:Comparison}) that is used to prove Theorem \ref{mainthm}. We then identify the associated $J_{A,B}$ minimization problem, which allows us to numerically compute the values of the constants $C(R,n)$ for a given weighted space (Section \ref{Sect:ProofMainRes}).

In the second half of the paper, we focus our attention on anti-Gauss space by considering the specific weight $a(x)=e^{|x|^2/2}$. We start by collecting relevant background information on confluent hypergeometric functions (Section \ref{Sec:CHF}). These functions arise naturally in studying the $J_{A,B}$ minimization problem and eigenvalue problem on centered balls in anti-Gauss space. We explicitly solve both of these problems in terms of confluent hypergeometric functions (Sections \ref{Section:Ball} and \ref{Section:ELSoln}), and we use our solutions to numerically approximate the constants $C(R,n)$. Our paper concludes with a numerical discussion of the anti-Gaussian constants $C(R,n)$ (Section \ref{Sect:NumLowBds}).

\section{Preliminaries and Notation}\label{Section:Notation}
Here we collect the relevant notation and definitions used throughout our paper. Throughout this section, $\Omega$ denotes a bounded $C^{\infty}$ domain and $\phi:\mathbb{R} \to \mathbb{R}$ denotes a function that is even, convex, and $C^{\infty}$ on $\mathbb{R}$ with $\phi(|x|)$ real analytic on $\mathbb{R}^n$.

\subsection*{The Laplacian with drift and weighted Sobolev spaces}
We denote $a(x)=e^{\phi(|x|)}$, take $\alpha$ to be the absolutely continuous measure with density $a$, and again write $\AF$ for the Laplace operator with drift:
\[
 \AF u=\frac{1}{a}\nabla\cdot(a\nabla u)=\Delta u+\nabla \phi \cdot \nabla u.
\]
We write
\[
d\alpha=a\,dx
\]
for the absolutely continuous measure on $\RR^n$ with density $a(x)$. The $L^p(\Omega,\alpha)$ norm is defined in the usual way:
\[
\|u\|_{L^p(\Omega,\alpha)}=\left(\int_\Omega |u|^p\,d\alpha\right)^{1/p}.
\]
We write $L^p(\Omega,\alpha)$ for the space of $\alpha$-measurable functions with finite norm.

Sobolev spaces and their norms for the weighted space $(\Omega,\alpha)$ are defined in the expected manner. For instance, $H^1(\Omega,\alpha)$ denotes the space of functions in $L^2(\Omega,\alpha)$ with weak first-order partials that also belong to $L^2(\Omega,\alpha)$. The relevant norm on $H^1(\Omega,\alpha)$ is
\[
\|u\|_{H^1(\Omega,\alpha)}=\left(\int_{\Omega}|\nabla u|^2\,d\alpha+\int_{\Omega}u^2\,d\alpha \right)^{1/2}.
\]
Similarly, $H^2(\Omega,\alpha)$ denotes the space of functions in $L^2(\Omega,\alpha)$ with weak first and second-order partials that belong to $L^2(\Omega,\alpha)$, with the norm given by
\[
\|u\|_{H^2(\Omega,\alpha)}=\left(\int_{\Omega}|D^2u|^2\,d\alpha+\int_{\Omega}|\nabla u|^2\,d\alpha+\int_{\Omega}u^2\,d\alpha \right)^{1/2}.
\]
If $C_c^{\infty}(\Omega)$ denotes the collection of smooth functions on $\Omega$ with compact support, then $H_0^1(\Omega,\alpha)$ and $H_0^2(\Omega,\alpha)$ denote the closures of $C_c^{\infty}(\Omega)$ in each of $H^1(\Omega,\alpha)$ and $H^2(\Omega,\alpha)$ with respect to the corresponding norms.

The Rayleigh Quotient for the eigenvalue problem \eqref{eqn:main pde problemma} is
\[
R[u]=\frac{\int_\Omega (\AF u)^2\,d\alpha}{\int_\Omega u^2\,d\alpha}.
\]

\subsection*{Symmetrization}

Write $\mathbb{B}_R=\{x\in \mathbb{R}^n:|x|<R\}$ for the centered ball of radius $R$ in $\mathbb{R}^n$, and define $\Phi:(0,\infty)\to\mathbb{R}$ via
\[
\Phi(R)=\alpha(\mathbb{B}_R)=\beta_n \int_0^Ra(r)r^{n-1}\,dr,
\]
where $\beta_n$ denotes the $(n-1)$-dimensional Hausdorff measure of the unit sphere $\mathbb{S}^{n-1}\subset \mathbb{R}^n$. We let $\Omega^{\#}$ denote the ball
\[
\Omega^{\#}=\mathbb{B}_R,
\]
where $R$ is chosen so that $\Phi(R)=\alpha(\Omega)$. Let $\textup{Per}$ denote the perimeter of a set weighted by $a(x)$. Then for all sufficiently regular sets $E\subseteq \mathbb{R}^n$, we have
\[
\textup{Per}(E)=\int_{\partial E}a(x)\,d\mathcal H^{n-1}(x).
\]
According to the isoperimetric inequality of Chambers \cite{C15}, we then have
\begin{equation}\label{Ineq:ChamIso}
\textup{Per}(\Omega^{\#})\leq \textup{Per}(\Omega).
\end{equation}

Given $u\in L^1(\Omega,\alpha)$, we write $\lambda_u:\mathbb{R} \to [0,\infty)$ for the distribution function of $u$:
\[
\lambda_u(t)=\alpha \left(\{x\in \Omega:u(x)>t\}\right).
\]
The decreasing rearrangement $u^{\ast}$ is defined on the interval $[0,\alpha(\Omega)]$ using the distribution function:
\[
u^{\ast}(t)=
\begin{cases}
\underset{\Omega}{\textup{ess sup}}\,u & \textup{if }t=0,\\
\inf \{s:\lambda_u(s)\leq t \} & \textup{if }0<t<\alpha(\Omega),\\
\underset{\Omega}{\textup{ess inf}}\,u & \textup{if }t=\alpha(\Omega).
\end{cases}
\]
Finally, the weighted symmetrization $u^{\#}$ is defined on $\Omega^{\#}$ in terms of the decreasing rearrangement:
\[
u^{\#}(x)=u^{\ast}(\Phi(\mathbb{B}_r)),
\]
where $r=|x|$. For more on rearrangements and symmetrization methods, see \cite{Ke2}.
\section{Existence of the Spectrum and Regularity of Solutions}
\label{coercsection}

The goal of this section is to show that the problem
\begin{equation}
  \begin{cases}\AF^2 u=\Lambda u &\text{in $\Omega$},\\
  u=0&\text{on $\partial\Omega$},\\ \label{eqn:PDE prob regularity}
  \frac{\partial u}{\partial \vn}=0&\text{on $\partial\Omega$},
  \end{cases}
\end{equation}
admits a sequence of eigenvalues that tends to infinity, and that eigenfunctions exhibit regularity up to the boundary. More precisely, we prove:

\begin{prop}\label{spect} Suppose $\Omega \subset \RR^n$ denotes a bounded $C^{\infty}$ domain and $a(x)=e^{\phi(x)}$, where $\phi(x)$ is a real analytic function on $\RR^n$. Let $\alpha$ denote the absolutely continuous measure on $\RR^n$ with density $a(x)$. Define a bilinear form by
\[
A(u,v)=\int_{\Omega}(\AF u)(\AF v)\,d\alpha,\qquad u,v\in H_0^2(\Omega,\alpha).
\]
Then the eigenvalues of the operator corresponding to the bilinear form $A$ have finite multiplicity and satisfy
\[
0\leq \Lambda_1\leq \Lambda_2\leq \cdots\leq \Lambda_n\leq \cdots \to\infty\qquad\text{as $n\to\infty$}.
\]
The eigenfunctions form a complete orthonormal basis of $L^2(\Omega,\alpha)$, are real analytic in $\Omega$, and are smooth on the closure $\overline{\Omega}$.
\end{prop}

\begin{proof}
We show the existence of positive constants $K_1, K_2$, and $K_3$ such that
\begin{equation} \label{Ineq: Coerc}
 K_1\|u\|_{H^2(\Omega,\alpha)}^2\leq A(u,u)+K_2\|u\|_{L^2(\Omega,\alpha)}^2\leq K_3\|u\|_{H^2(\Omega,\alpha)}^2
\end{equation}
for all functions $u\in H^2_0(\Omega,\alpha)$. First assume $u\in C_c^{\infty}(\Omega)$. Integrating by parts, we have
\begin{equation} \label{Eqn:IBP Coerc}
A(u,u)=-\sum_{i,j}\int_{\Omega}u_{x_i}u_{x_j}\phi_{x_ix_j}a\,dx+\|D^2u\|^2_{L^2(\Omega,\alpha)}.
\end{equation}
Let $M$ denote a constant such that $|\phi_{x_ix_j}|\leq M$ on $\overline{\Omega}$ for every $i,j$. Then by standard estimates,
\begin{equation} \label{Ineq:CS Coerc}
\left |\sum_{i,j}\int_{\Omega}u_{x_i}u_{x_j}\phi_{x_ix_j}a\,dx \right| \leq n M \| \nabla u \|^2_{L^2(\Omega,\alpha)}.
\end{equation}
By the proof of Theorem 7.27 in \cite{GT}, for each $0<\varepsilon<\frac{1}{2}$, there exists a constant $C_1=C_1(\varepsilon)$ so that for all $u\in C_c^{\infty}(\Omega)$, we have
\begin{equation*}
\| \nabla u\|^2_{L^2(\Omega)} \leq \varepsilon \|u\|^2_{H^2(\Omega)}+\frac{C_1}{\varepsilon}\|u\|^2_{L^2(\Omega)}.
\end{equation*}
This inequality implies
\begin{equation} \label{Ineq:GT Coerc}
\|D^2u\|^2_{L^2(\Omega)}\geq \left(\frac{1}{\varepsilon}-1\right)\| \nabla u \|^2_{L^2(\Omega)}-\left(\frac{C_1}{\varepsilon^2}+1\right)\|u\|^2_{L^2(\Omega)}.
\end{equation}
Next, let $C_2$ and $C_3$ be constants with $0<C_2\leq a(x)\leq C_3$ on $\overline{\Omega}$. We 
combine \eqref{Eqn:IBP Coerc} with estimates  \eqref{Ineq:CS Coerc} and \eqref{Ineq:GT Coerc}. For any $\delta, K_2>0$ we have
\begin{align*}
& A(u,u)+K_2\|u\|^2_{L^2(\Omega,\alpha)} \\
&  \geq \|D^2u\|^2_{L^2(\Omega,\alpha)}-nM\| \nabla u\|^2_{L^2(\Omega,\alpha)}+K_2\|u\|^2_{L^2(\Omega,\alpha)} \\
& \geq C_2 \left( \|D^2u\|^2_{L^2(\Omega)}-\frac{nMC_3}{C_2}\| \nabla u\|^2_{L^2(\Omega)}+K_2\|u\|^2_{L^2(\Omega)}\right)\\
& \geq C_2  \left( (1-\delta) \|D^2u\|^2_{L^2(\Omega)} + \left(\frac{\delta}{\varepsilon}-\delta-\frac{nMC_3}{C_2}\right)\| \nabla u\|^2_{L^2(\Omega)} + \left(K_2-\frac{C_1\delta}{\varepsilon^2}-\delta \right)\|u\|^2_{L^2(\Omega)}\right) \\
& \geq \frac{C_2}{C_3}  \left( (1-\delta) \|D^2u\|^2_{L^2(\Omega,\alpha)} + \left(\frac{\delta}{\varepsilon}-\delta-\frac{nMC_3}{C_2}\right)\| \nabla u\|^2_{L^2(\Omega,\alpha)} + \left(K_2-\frac{C_1\delta}{\varepsilon^2}-\delta \right)\|u\|^2_{L^2(\Omega,\alpha)}\right)\\
& \geq K_1 \|u\|^2_{H^2(\Omega,\alpha)},
\end{align*}
where we take $K_1=\frac{C_2}{C_3}\min \left \{ 1-\delta, \frac{\delta}{\varepsilon}-\delta-\frac{nMC_3}{C_2}, K_2-\frac{C_1\delta}{\varepsilon^2}-\delta \right \}$. Note that the second-to-last inequality assumes $K_1$ is positive. However, it is easy to see that given $0<\delta<1$, we may choose positive values for $\varepsilon$ and $K_2$ to ensure $K_1$ is indeed positive. With this value of $K_2$, we again use \eqref{Ineq:CS Coerc} to estimate
\begin{align*}
A(u,u) + K_2\|u\|^2_{L^2(\Omega,\alpha)} & \leq nM\| \nabla u \|^2_{L^2(\Omega,\alpha)}+\|D^2u\|^2_{L^2(\Omega,\alpha)}+K_2\|u\|^2_{L^2(\Omega,\alpha)}\\
& \leq K_3 \|u\|^2_{H^2(\Omega,\alpha)},
\end{align*}
where $K_3=\max \{1, nM, K_2\}$. Both inequalities of \eqref{Ineq: Coerc} are now established for $u\in C_c^{\infty}(\Omega)$. The same inequalities hold for general $u\in H^2_0(\Omega,\alpha)$ by approximation.

Next, the space $\domain$ is compactly embedded in $L^2(\Omega,\alpha)$ since the analogous result holds true for unweighted spaces (see Theorem 1 of \cite[p. 288]{E} and the remark following its proof). By \cite[Corollary 7.8, p. 88]{SH77}, the bilinear form $A$ has a set of weak eigenfunctions that form an orthonormal basis of $L^2(\Omega, \alpha)$. By the same result, eigenspaces are finite dimensional and the sequence of eigenvalues is bounded below and increasing to infinity. Nonnegativity of the spectrum follows immediately from inspecting the Rayleigh quotient.

 Regularity of the eigenfunctions follows from standard regularity results \cite[p. 668]{Nir55}, the Trace Theorem  \cite[Prop 4.3, p. 286 and Prop 4.5, p. 287]{taylor}, and the Analyticity Theorem \cite[p. 136]{BJS}. Weak eigenfunctions therefore solve problem \eqref{eqn:PDE prob regularity} in the classical sense.

\end{proof}

\section{Symmetrization}\label{Section:Comparison}
Throughout this section, $\Lambda_1$ denotes the lowest eigenvalue and $u$ denotes a corresponding eigenfunction for the problem 
\begin{equation}
\begin{cases}
\AF ^{2}u=\Lambda u&\text{in $\Omega$},\\
u=0\ &\text{on $\partial\Omega$,}\\ \label{eqn:PDE problem sym}
\frac{\partial u}{\partial \nu}=0\ &\text{on $\partial\Omega$.}
\end{cases}
\end{equation}
Recall from Theorem~\ref{spect} that the eigenfunctions are real analytic in $\Omega$, and are smooth on the closure $\overline{\Omega}$. In general, eigenfunctions change sign, so our next goal is to establish a comparison result for the positive and negative parts of $u$.  The argument below parallels Talenti's argument for the Euclidean clamped plate. See \cite{T81} and also \cite{CL16} for the analogous argument in Gauss space (where unbounded domains are considered).

At this point the reader might find it useful to review the notation and definitions of Section \ref{Section:Notation}.
\begin{prop}\label{Prop:Symm Rslt}
Suppose $\Omega\subset\RR^n$ is a bounded $C^\infty$ domain and $u$ is a principal eigenfunction as above. Assume $\phi$ is an even, convex $C^{\infty}$ function defined on $\mathbb{R}$ for which $\phi(|x|)$\footnote{Here and in the remainder of the paper we will abuse notation, writing $\phi$ for the both the initial function defined on $\mathbb{R}$ and also its radial extension to all of $\mathbb{R}^n$.} defines a radial real analytic function on $\RR^n$. Write $a(x)=e^{\phi(|x|)}$. Then there are constants $A, B\geq0$ satisfying
\[
\alpha(\BB_A)+\alpha(\BB_B)=\alpha(\Omega)
\]
and functions $v\in H^1_0(\BB_A,\alpha)\cap H^2(\BB_A,\alpha)$, $w\in H^1_0(\BB_B,\alpha)\cap H^2(\BB_B,\alpha)$ where
\[
\upsh\leq v(x) \text{ in } \BB_A \quad \text{and} \quad \umsh\leq w(x) \text{ in } \BB_B.
\]
Here $\upsh$ and $\umsh$ denote the weighted symmetrizations of the positive and negative parts of $u$, respectively, as defined in Section $\ref{Section:Notation}$. In addition, the functions $u, v$, and $w$ satisfy
\[
\frac{\int_{\Omega}(\AF u)^2 \,d\alpha}{\int_{\Omega}u^2\,d\alpha}\geq \frac{\int_{\BB_A}(\AF v)^2\,d\alpha+\int_{\BB_B}(\AF w)^2\,d\alpha}{\int_{\BB_A}v^2\,d\alpha+\int_{\BB_B}w^2\,d\alpha},
\]
and
\[
\frac{\partial v}{\partial r}(A)A^{n-1}a(A)=\frac{\partial w}{\partial r}(B)B^{n-1}a(B).
\]
\end{prop}

\begin{proof}
For $t\in\RR$, let $\gamma(t)=\alpha \left(\{x\in \RR^n:u(x)>t\}\right)$ denote the distribution function of $u$. By Cauchy-Schwarz, we have
\[
\left( \frac{1}{h}\int_{\{t<u\leq t+h\}}|\nabla u|\,d\alpha \right)^2 \leq \frac{\gamma(t)-\gamma(t+h)}{h}\cdot \frac{1}{h}\int_{\{t<u\leq t+h\}}|\nabla u|^2\,d\alpha.
\]
Invoking the coarea formula twice and taking the limit as $h\to 0$, we deduce
\begin{equation}\label{Ineq:PerBd}
\left( \textup{Per}\left( \{u>t\}\right)\right)^2\leq -\gamma'(t)\int_{\{u=t\}}|\nabla u|a\,d\HH^{n-1}.
\end{equation}
Here we have used (and in what follows we continue to use) the identity
\[
\partial \{u>t\}=\{u=t\}
\]
which holds for almost every $t$ by Sard's Theorem. By the Divergence Theorem, we have
\begin{equation}\label{Eqn:DivThmAu}
\int_{\{u=t\}}|\nabla u|a\,d\HH^{n-1}=\int_{\{u>t\}}-\AF u\,d \alpha.
\end{equation}
Combining \eqref{Ineq:ChamIso} with \eqref{Ineq:PerBd} and \eqref{Eqn:DivThmAu}, we deduce 
\begin{equation} \label{eqn:Per1}
\left(\textup{Per}\left(\{u^{\#}>t\}\right)\right)^2  \leq -\gamma'(t)\int_{\{u>t\}}-\AF u\,d\alpha \leq -\gamma'(t)\int_0^{\gamma(t)}\left[(\AF u)_{-}^{\ast}(s)-(\AF u)_{+}^{\ast}(\alpha(\Omega)-s)\right]\,ds,
\end{equation}
where $(\AF u)_{+}^{\ast}$ and $(\AF u)_{-}^{\ast}$ denote the decreasing rearrangements of the positive and negative parts of $\AF u$, respectively.

Define
\[
f(s)=(\AF u)_{-}^{\ast}(s)-(\AF u)_{+}^{\ast}(\alpha(\Omega)-s),\quad 0\leq s\leq \alpha(\Omega),
\]
and suppose $\rho=\rho(t)$ satisfies
\[
\{u^{\#}>t\}=\BB_\rho.
\]
We then have $\Phi(\rho)=\gamma(t)$, which implies $\rho=\Phi^{-1}(\gamma(t))$. The perimeter of $\{u^{\#}>t\}$ then satisfies
\begin{equation}\label{eqn:Per2}
\textup{Per}(\{u^{\#}>t\})=\int_{\partial \BB_\rho}a\,d\HH^{n-1}=\beta_n\rho^{n-1}a(\rho).
\end{equation}
Combining \eqref{eqn:Per1} and \eqref{eqn:Per2}, we have
\[
1\leq -\frac{1}{\beta_n^2}\rho^{2-2n}a(\rho)^{-2}\gamma'(t)\int_0^{\gamma(t)}f(s)\,ds,
\]
and integrating from $0$ to $t$ yields
\[
t\leq -\frac{1}{\beta_n^2}\int_0^t\left(\Phi^{-1}(\gamma(s))^{2-2n}a(\Phi^{-1}(\gamma(s)))^{-2}\gamma'(s)\int_0^{\gamma(s)}f(q)\,dq\right)\,ds.
\]
Making the change of variable $z=\gamma(s)$ and letting $t=u_+^{\#}(x)$ yields
\[
u_+^{\#}(x)\leq \frac{1}{\beta_n^2}\int_{\Phi(r)}^{\gamma(0)}\left(\Phi^{-1}(z)^{2-2n}a(\Phi^{-1}(z))^{-2}\int_0^zf(q)\,dq\right)\,dz,
\]
where $r=|x|$. Define a number $A$ satisfying $\Phi(A)=\gamma(0)$ and define a radial function $v(r)$ on $\BB_A$ using the formula
\[
v(r)=\frac{1}{\beta_n^2}\int_{\Phi(r)}^{\Phi(A)}\left(\Phi^{-1}(z)^{2-2n}a(\Phi^{-1}(z))^{-2}\int_0^zf(q)\,dq\right)\,dz.
\]
By construction, the functions $\upsh$ and $v$ satisfy the inequality
\begin{equation}\label{eqn:Comp1}
u_+^{\#}\leq v \textup{ in } \BB_A.
\end{equation}
A straightforward calculation shows that $v$ solves the problem
\begin{equation*}
\begin{cases}
-\AF v =f\circ \Phi &\text{in $\BB_A$},\\
v=0\ &\text{on $\partial \BB_A$.}
\end{cases}
\end{equation*}

Letting $y=-u$, we see that $y$ is also a principal eigenfunction for problem \eqref{eqn:PDE problem sym}. Define $h(s)$ on $[0,\alpha(\Omega)]$ via
\begin{align*}
h(s)&=(\AF y)_-^{\ast}(s)-(\AF y)_+^{\ast}(\alpha(\Omega)-s)\\
&=(\AF u)_{+}^{\ast}(s)-(\AF u)_-^{\ast}(\alpha(\Omega)-s)\\
&=-f(\alpha(\Omega)-s).
\end{align*}
Define a number $B$ satisfying $\Phi(B)=\alpha(\{y>0\})=\alpha(\{u<0\})$ and define a radial function $w(r)$ on $\BB_B$ using the formula
\[
w(r)=\frac{1}{\beta_n^2}\int_{\Phi(r)}^{\Phi(B)}\left(\Phi^{-1}(z)^{2-2n}a(\Phi^{-1}(z))^{-2}\int_0^zh(q)\,dq\right)\,dz.
\]
Then as above, $w$ solves the problem
\begin{equation*}
\begin{cases}
-\mathcal{A}w =h\circ \Phi &\text{in $\BB_B$},\\
w=0\ &\text{on $\partial \BB_B$},
\end{cases}
\end{equation*}
and $u_-^{\#}$ and $w$ satisfy
\begin{equation}\label{eqn:Comp2}
u_-^{\#}\leq w \textup{ in }\BB_B.
\end{equation}
Next, note that the Divergence Theorem implies
\begin{equation}\label{eqn:Auintegral}
\int_{\Omega}\AF u\,d\alpha=\int_{\partial \Omega}a\nabla u \cdot \nu\,d\HH^{n-1}=0.
\end{equation}
By construction, the numbers $A$ and $B$ satisfy
\[
\Phi(A)+\Phi(B)=\alpha(\Omega).
\]
Combining with \eqref{eqn:Auintegral}, we have
\[
\int_{\BB_A}-\AF v\,d\alpha=\int_{\BB_A}f\circ \Phi\,d\alpha=\int_0^{\Phi(A)}f(s)\,ds=-\int_0^{\Phi(B)}f(\alpha(\Omega)-s)\,ds=\int_{\BB_B}-\AF w\,d\alpha.
\]
On the other hand, by the Divergence Theorem, we see
\[
\int_{\BB_A}\AF v\,d\alpha=\int_{\partial \BB_A}a\nabla v\cdot \nu \,d\HH^{n-1}=\beta_n \frac{\partial v}{\partial r}(A)A^{n-1}a(A).
\]
We conclude that
\[
\frac{\partial v}{\partial r}(A)A^{n-1}a(A)=\frac{\partial w}{\partial r}(B)B^{n-1}a(B).
\]
We next compute
\begin{align*}
\int_{\BB_A}(\AF v)^2\,d\alpha+\int_{\BB_B}(\AF w)^2\,d\alpha & =\int_{\BB_A}\left[(\AF u)_-^{\ast}(\Phi(r))-(\AF u)_+^{\ast}(\alpha(\Omega)-\Phi(r))\right]^2\,d\alpha\\
&\quad +\int_{\BB_B}\left[(\AF u)_+^{\ast}(\Phi(r))-(\AF u)_-^{\ast}(\alpha(\Omega)-\Phi(r))\right]^2\,d\alpha\\
&= \int_0^{\Phi(A)}\left[(\AF u)_-^{\ast}(s)-(\AF u)_+^{\ast}(\alpha(\Omega)-s)\right]^2\,ds\\
&\quad + \int_0^{\Phi(B)}\left[(\AF u)_+^{\ast}(s)-(\AF u)_-^{\ast}(\alpha(\Omega)-s)\right]^2\,ds\\
&=\int_0^{\alpha(\Omega)}\left[(\AF u)_-^{\ast}(s)-(\AF u)_+^{\ast}(\alpha(\Omega)-s)\right]^2\,ds.
\end{align*}
The functions $(\AF u)_+^{\ast}(s)$ and $(\AF u)_-^{\ast}(\alpha(\Omega)-s)$ are never simultaneously nonzero, and we therefore deduce
\[
\int_{\BB_A}(\AF v)^2\,d\alpha+\int_{\BB_B}(\AF w)^2\,d\alpha=\int_0^{\alpha(\Omega)}[(\AF u)_-^{\ast}(s)^2+(\AF u)_+^{\ast}(s)^2]\,ds=\int_{\Omega}(\AF u)^2\,d\alpha.
\]
Combining the equality immediately above with inequalities \eqref{eqn:Comp1} and \eqref{eqn:Comp2}, we conclude
\[
\Lambda_1=\frac{\int_{\Omega}(\AF u)^2\,d\alpha}{\int_{\Omega}u^2\,d\alpha}\geq \frac{\int_{\BB_A}(\AF v)^2\,d\alpha+\int_{\BB_B}(\AF w)^2\,d\alpha}{\int_{\BB_A}v^2\,d\alpha+\int_{\BB_B}w^2\,d\alpha}.
\]
A straightforward calculation shows that both $\left(\frac{\partial v}{\partial r}\right)^2r^{n-1}$ and $\left(\frac{\partial^2 v}{\partial r^2}\right)^2r^{n-1}$ remain bounded as $r\to 0$, so that $\frac{\partial v}{\partial r}, \frac{\partial^2 v}{\partial r^2}\in L^2(\BB_A,\alpha)$. Since $-\AF v=f\circ \Phi \in L^2(\BB_A,\alpha)$, it follows that $v\in H_0^1(\BB_A,\alpha)\cap H^2(\BB_A,\alpha)$ and similarly $w\in H_0^1(\BB_B,\alpha)\cap H^2(\BB_B,\alpha)$.
\end{proof}

\section{Proof of Main Result, the $J_{A,B}$ minimization problem, and further consequences with log-convex density}\label{Sect:ProofMainRes}

We begin this section with a proof of Theorem \ref{mainthm}:
\begin{proof}
Let $u$ be a principal eigenfunction for $\Lambda_1(\Omega)$. Then by Proposition \ref{Prop:Symm Rslt}, there exist numbers $A$ and $B$ satisfying $\alpha(\BB_A)+\alpha(\BB_B)=\alpha(\Omega)$ and radial functions functions $v\in H^1_0(\BB_A,\alpha)\cap H^2(\BB_A,\alpha), w\in H^1_0(\BB_B,\alpha)\cap H^2(\BB_B,\alpha)$ where
\begin{equation}\label{eqn:UBonJAB}
\Lambda_1(\Omega)=\frac{\int_{\Omega}(\AF u)^2 \,d\alpha}{\int_{\Omega}u^2\,d\alpha}\geq \frac{\int_{\BB_A}(\AF v)^2\,d\alpha+\int_{\BB_B}(\AF w)^2\,d\alpha}{\int_{\BB_A}v^2\,d\alpha+\int_{\BB_B}w^2\,d\alpha}.
\end{equation}
Define
\begin{equation}\label{Eqn:JAB Def}
J_{A,B}=\inf_{(v,w)} \left \{ \frac{\int_{\BB_A}(\AF v)^2\,d\alpha+\int_{\BB_B}(\AF w)^2\,d\alpha}{\int_{\BB_A}v^2\,d\alpha+\int_{\BB_B}w^2\,d\alpha} \right \},
\end{equation}
where the $\inf$ is taken over all pairs of radial functions $(v,w)$ with $v\in H^1_0(\BB_A,\alpha)\cap H^2(\BB_A,\alpha), w\in H^1_0(\BB_B,\alpha)\cap H^2(\BB_B,\alpha)$ satisfying
\begin{equation}\label{JAB bdry cond}
\frac{\partial v}{\partial r}(A)A^{n-1}a(A)=\frac{\partial w}{\partial r}(B)B^{n-1}a(B).
\end{equation}
We therefore have
\begin{align*}
\Lambda_1(\Omega) & \geq J_{A,B}\\
& \geq \inf \{ J_{A,B}: \alpha(\BB_A)+\alpha(\BB_B)=\alpha(\Omega)\}\\
&= \left( \frac{\inf \{ J_{A,B}: \alpha(\BB_A)+\alpha(\BB_B)=\alpha(\Omega)\}}{\Lambda_1(\Omega^{\#})}\right) \Lambda_1(\Omega^{\#}).
\end{align*}
We take the constant $C$ to be
\begin{equation}
C= \frac{\inf \{ J_{A,B}: \alpha(\BB_A)+\alpha(\BB_B)=\alpha(\Omega)\}}{\Lambda_1(\Omega^{\#})} \label{eqn:Cdef}.
\end{equation}
 Note that $C>0$ by Proposition~\ref{Prop:MinExt}, since the infimum is attained by $v,w$ in the appropriate spaces. We have $C\leq1$ by the earlier relation \eqref{eqn:UBonJAB}. 

\end{proof}

We now proceed with our detailed study of $J_{A,B}$ as defined in \eqref{Eqn:JAB Def}. Our first goal is to show that $J_{A,B}$ is achieved by an appropriate pair of functions $(v,w)$.

\begin{prop}\label{Prop:MinExt} When $n\geq3$ or $n=2$ and the fundamental mode is radial, the infimum defining $J_{A,B}$ in \eqref{Eqn:JAB Def} is achieved by some pair of functions $(v,w)$. Moreover, a minimizing pair $(v,w)$ satisfies
\begin{align*}
\AF^2 v&=\mu v \quad \textup{in }\BB_A,\\
\AF^2w&=\mu w \quad \textup{in }\BB_B, 
\end{align*}
where $\mu=J_{A,B}$, together with the natural boundary condition
\[
\AF v(A)+\AF w(B)=0.
\]
\end{prop}

\begin{proof}
We first observe that if either $A$ or $B$ equals zero, then $J_{A,B}=\Lambda_1(\Omega^{\#})$ is achieved by any principal eigenfunction (which is radial by Proposition \ref{Thm:RadMod} for dimensions $n\geq3$, and by assumption for $n=2$). In what follows, we therefore assume that $A$ and $B$ are positive. We begin by noting that
\[
J_{A,B}=\inf_{(v,w)} \left \{ \int_{\BB_A}(\AF v)^2\,d\alpha+\int_{\BB_B}(\AF w)^2\,d\alpha \right \},
\]
where the $\inf$ is taken over all pairs of radial functions $(v,w)$ with $v\in H^1_0(\BB_A,\alpha)\cap H^2(\BB_A,\alpha), w\in H^1_0(\BB_B,\alpha)\cap H^2(\BB_B,\alpha)$ satisfying \eqref{JAB bdry cond} together with the normalization
\[
\int_{\BB_A}v^2\,d\alpha+\int_{\BB_B}w^2\,d\alpha=1.
\]
Let $(v_m,w_m)$ denote such a sequence of normalized functions where
\begin{equation}\label{conv to JAB}
\int_{\BB_A}(\AF v_m)^2\,d\alpha+\int_{\BB_B}(\AF w_m)^2\,d\alpha \to J_{A,B}.
\end{equation}
We claim that the sequences $v_m$ and $w_m$ are bounded in $H^2(\BB_A,\alpha)$ and $H^2(\BB_B,\alpha)$, respectively. Define a new sequence of radial functions on $\BB_A$ using the formula
\[
\tilde{v}_m(r)=v_m(r)-\left(\frac{r^2-A^2}{2A}\right)\frac{\partial v_m}{\partial r}(A),
\]
and observe that $\tilde v_m \in H^2_0(\BB_A,\alpha)$. From Proposition \ref{spect}, there exist positive constants $K_1$ and $K_2$ with
\begin{equation}\label{Ineq:bdness of tilde}
\|\tilde v_m\|^2_{H^2(\BB_A,\alpha)}\leq \frac{1}{K_1}\left( \int_{\BB_A}(\AF \tilde v_m)^2\,d\alpha+K_2\|\tilde v_m\|^2_{L^2(\BB_A,\alpha)}\right).
\end{equation}
If we can show that the derivatives $\frac{\partial v_m}{\partial r}(A)$ form a bounded sequence, then \eqref{Ineq:bdness of tilde} will give that $v_m$ are bounded in $H^2(\BB_A,\alpha)$. To establish boundedness of $\frac{\partial v_m}{\partial r}(A)$, observe that $\int_{\BB_A} \AF v_m \,dx$ are bounded. Writing $\AF$ in spherical coordinates, we see that
\[
\AF = \frac{1}{r^{n-1}}\frac{\partial }{\partial r}\left(r^{n-1}\frac{\partial}{\partial r}\right)+\frac{\partial \phi}{\partial r}\frac{\partial}{\partial r}+\frac{1}{r^2}\Delta_S=\frac{\partial^2}{\partial r^2}+\left(\frac{n-1}{r}+\frac{\partial \phi}{\partial r}\right)\frac{\partial}{\partial r}+\frac{1}{r^2}\Delta_S,
\]
where $\Delta_S$ denotes the spherical Laplacian. Integrating by parts, we then compute
\begin{align*}
\int_{\BB_A} \AF v_m \,dx&=\beta_n\int_0^A\frac{\partial}{\partial r}\left(r^{n-1}\frac{\partial v_m}{\partial r}\right)\,dr+\beta_n\int_0^A\frac{\partial \phi}{\partial r}\frac{\partial v_m}{\partial r}r^{n-1}\,dr\\
&=\beta_nA^{n-1}\frac{\partial v_m}{\partial r}(A)-\beta_n\int_0^Av_m\left(\frac{\partial^2 \phi}{\partial r^2}+\frac{n-1}{r}\frac{\partial \phi}{\partial r}\right)r^{n-1}\,dr.
\end{align*}
Because $\phi(r)$ is even, it follows that $\frac{\partial \phi}{\partial r}(0)=0$, and so the expression $\frac{\partial^2 \phi}{\partial r^2}+\frac{n-1}{r}\frac{\partial \phi}{\partial r}$ remains bounded over the interval $[0,A]$. Since $\int_{\BB_A}v_m\,dx$ are bounded, we conclude $\frac{\partial v_m}{\partial r}(A)$ are bounded as well.

Having established that the sequence $\|v_m\|_{H^2(\BB_A,\alpha)}$ is bounded, by Banach-Alaoglu we may pass to a subsequence and assume that $v_m\to v$ weakly for some $v\in H^2(\BB_A,\alpha)$. Similarly, we may assume $w_m \to w$ weakly for some $w\in H^2(\BB_B,\alpha)$. Starting with the inequality $\int_{\BB_A}\left( \AF(v_m-v)\right)^2\,d\alpha \geq 0$, we have
\[
\int_{\BB_A}\left(\AF v_m\right)^2\,d\alpha \geq 2\int_{\BB_A}\left(\AF v_m\right) \left(\AF v\right)\,d\alpha - \int_{\BB_A}\left(\AF v\right)^2\,d\alpha.
\]
Taking the $\liminf$ and using weak convergence, we see
\[
\liminf_{m\to \infty} \int_{\BB_A}\left(\AF v_m\right)^2\,d\alpha \geq \int_{\BB_A}\left(\AF v \right)^2\,d\alpha.
\]
We similarly deduce
\[
\liminf_{m\to \infty} \int_{\BB_B}\left(\AF w_m\right)^2\,d\alpha \geq \int_{\BB_B}\left(\AF w \right)^2\,d\alpha.
\]
Combining the two inequalities immediately above with \eqref{conv to JAB}, we have
\[
J_{A,B}\geq \int_{\BB_A}\left(\AF v \right)^2\,d\alpha + \int_{\BB_B}\left(\AF w \right)^2\,d\alpha.
\]
By the Rellich-Kondrachov Theorem, we may pass to further subsequences and assume
\[
v_m\to v\quad \textup{in }L^2(\BB_A,\alpha)\qquad \textup{and} \qquad w_m\to w\quad \textup{in }L^2(\BB_B,\alpha).
\]
It follows that
\[
\int_{\BB_A}v^2\,d\alpha+\int_{\BB_B}w^2\,d\alpha=1.
\]
By Part III of Theorem $6.3$ in \cite{AF03}, $H^2(\BB_A, \alpha)$ embeds compactly in $C^1(\overline{\BB_A})$ and similarly for the ball $\BB_B$. Hence $v$ and $w$ are radial and by passing to a subsequence, \eqref{JAB bdry cond} is satisfied, and $v\in H^1_0(\BB_A,\alpha)$, $w\in H^1_0(\BB_B,\alpha)$.

By the usual calculus of variations argument, $v$ and $w$ satisfy weak eigenvalue equations. Standard elliptic regularity results \cite[Theorem 2]{Nir55} thus imply that $v$ and $w$ are $C^{\infty}$ on $\mathbb{B}_A$ and $\mathbb{B}_B$, respectively. We now write $\mu=J_{A,B}$ and assume that $\psi \in C^2(\overline \BB_A)$, $\zeta \in  C^2(\overline \BB_B)$ are radial functions vanishing at $r=A$ and $r=B$, satisfying
\[
\frac{\partial \psi}{\partial r}(A)A^{n-1}a(A)=\frac{\partial \zeta}{\partial r}(B)B^{n-1}a(B).
\]
Then
\begin{align*}
0&=\int_{\BB_A}\left(\AF^2 v-\mu v\right)\psi\,d\alpha+\int_{\BB_B}\left(\AF^2w-\mu w\right)\zeta \,d\alpha\\
 &+\beta_n(\AF v)(A)\frac{\partial \psi}{\partial r}(A)A^{n-1}a(A)+\beta_n(\AF w)(B)\frac{\partial \zeta}{\partial r}(B)B^{n-1}a(B).
\end{align*}
Taking $\psi=0$ and $\zeta=0$ separately, we deduce $\AF^2v=\mu v$ and $\AF^2w=\mu w$ together with the natural boundary condition
\[
\AF v(A)+\AF w(B)=0.
\]
\end{proof}

\begin{remark}\label{Rmk:JABpos}
We note that $J_{A,B}>0$. To understand why this is the case, assume $J_{A,B}=0$ and take $(v,w)$ to be a minimizing pair of functions as in Proposition \ref{Prop:MinExt}. Since $J_{A,B}=0$, we have $\AF v=0$, so in particular $\int_{\BB_A}v\AF v\,dx=0$. Note that
\begin{equation}\label{Eqn:JP1}
0=\int_{\BB_A}v\AF v\,dx=\beta_n\int_0^Av\left(\frac{1}{r^{n-1}}\frac{\partial}{\partial r}\left(r^{n-1}\frac{\partial v}{\partial r}\right)+\frac{\partial \phi}{\partial r}\frac{\partial v}{\partial r}\right)r^{n-1}\,dr.
\end{equation}
Since $v(A)=0$, integration by parts gives
\begin{equation}\label{Eqn:JP2}
\int_0^Av\frac{\partial }{\partial r}\left(r^{n-1}\frac{\partial v}{\partial r}\right)\,dr=-\int_0^A\left(\frac{\partial v}{\partial r}\right)^2r^{n-1}\,dr.
\end{equation}
Integrating by parts once again and using that $v(A)=0$, we see
\[
\int_0^Av\frac{\partial \phi}{\partial r}\frac{\partial v}{\partial r}r^{n-1}\,dr=-\int_0^Av\left(\frac{\partial v}{\partial r}\frac{\partial \phi}{\partial r}+v\frac{\partial^2\phi}{\partial r^2}+\frac{n-1}{r}v\frac{\partial \phi}{\partial r}\right)r^{n-1}\,dr\\
\]
and so
\begin{equation}\label{Eqn:JP3}
\int_0^Av\frac{\partial \phi}{\partial r}\frac{\partial v}{\partial r}r^{n-1}\,dr=-\frac{1}{2}\int_0^Av^2\left(\frac{\partial^2\phi}{\partial r^2}+\frac{n-1}{r}\frac{\partial \phi}{\partial r}\right)r^{n-1}dr.
\end{equation}
Combining \eqref{Eqn:JP1}, \eqref{Eqn:JP2}, and \eqref{Eqn:JP3} we see
\[
0=\int_0^A\left(\frac{\partial v}{\partial r}\right)^2r^{n-1}\,dr+\frac{1}{2}\int_0^Av^2\left(\frac{\partial^2\phi}{\partial r^2}+\frac{n-1}{r}\frac{\partial \phi}{\partial r}\right)r^{n-1}dr.
\]
By assumption, $\phi$ is convex, and since $\phi(r)$ is even, we have $\frac{\partial \phi}{\partial r}(0)=0$. We conclude that $v=0$ which violates the normalization assumption $\int_{\BB_A}v^2\,d\alpha=1$. We deduce that $J_{A,B}$ must be positive.
\end{remark}

In the next theorem, we describe the form of the eigenfunctions of the ball and identify the ground state.

\begin{prop}\label{Thm:RadMod} Let $a(x)$ be as in Proposition \ref{Prop:Symm Rslt}. If $R>0$, then for all dimensions $n$, the solutions to the eigenvalue problem
\begin{equation}\label{Eqn:PDEProbBall1}
\begin{cases}
  \AF^2u=\Lambda u &\text{in $\BB_R$,}\\
  u=\frac{\partial u}{\partial r}=0 &\text{when $|x|=R$,}
 \end{cases}
 \end{equation}
 take the form
\[
u(r,\that)=y(r)Y_l(\that), \qquad r\in (0,R), \ \that \in \mathbb{S}^{n-1},
\]
where $Y_l$ is a spherical harmonic of order $l$ and $y$ is the radial function satisfying the radial eigenvalue problem 
\begin{equation}\label{Eqn:ODEy1}
 \begin{cases}\left(\frac{\partial^2}{\partial r^2}+\left(\frac{n-1}{r}+\frac{\partial \phi}{\partial r}\right)\frac{\partial}{\partial r}-\frac{l(l+n-2)}{r^2}\right)^2y=\Lambda y \qquad\text{when $r\in(0,R)$,}\\
  y(R)=y'(R)=0.
 \end{cases}
\end{equation}

Furthermore, for all dimensions $n\geq 3$ the fundamental mode of the ball $\BB_R$ corresponds to $l=0$ and is purely radial.
\end{prop}

\begin{remark} In dimension $n=2$, we expect the fundamental mode of the ball also corresponds to $l=0$, but we cannot prove this claim with our methods. The work in the following proof shows that $N[u]$ is increasing in $l$ for $l\geq1$ and all $n\geq 2$, and hence the ground state is obtained when either $l=0$ or $l=1$. For many drum and plate eigenvalue problems, eigenfunctions with more nodes of the same type correspond to higher eigenvalues. In the anti-Gaussian setting, eigenfunctions corresponding to the lowest $l=0$ eigenvalue have no angular nodes, while eigenfunctions corresponding to the lowest $l=1$ eigenvalue have one angular node. We also have numerical evidence that the $n=2$ fundamental mode corresponds to $l=0$; see Figure~\ref{fig:n=0,l=0,1} in Section~\ref{Section:Ball}.
\end{remark}

\begin{proof} 
We start by considering the numerator of the Rayleigh quotient, which we denote $N[u]$:
\[
N[u]:=\int_\Omega(\AF u)^2\,d\alpha.
\]
We observe that the operator $\Delta_S$ commutes with $\AF$ and hence with $\AF^2$. By Proposition~\ref{spect}, each eigenvalue $\Lambda$ has finite multiplicity and so its corresponding eigenspace $X_\Lambda\subseteq H^2_0(\BB_R,\alpha)$ is finite-dimensional. Thus, by the standard argument, the operators $\AF^2$ and $\Delta_S$ are simultaneously diagonalizable. The spherical harmonics $Y_l(\that)$ are eigenfunctions of the spherical Laplacian, so we may write the eigenfunctions of problem \eqref{Eqn:PDEProbBall1} in the form $u(r,\that)=y(r)Y_l(\theta)$. Expressing $\AF$ in spherical coordinates, the eigenvalue problem \eqref{Eqn:PDEProbBall1} reduces to \eqref{Eqn:ODEy1}.

We may assume the $Y_l$ are normalized over $L^2(\Ss^{n-1},\HH^{n-1})$ and the eigenfunction $u$ is normalized over $L^2(\BB_R,\alpha)$. Writing $k=l(l+n-2)$, we have $\Delta_SY_l=-k Y_l$. In the following calculations, we make use of shorthand notation, writing $y'(r)$ in place of $\partial y/\partial r$. We therefore have
\[
 N[u]=\int_0^R\int_{\Ss^{n-1}}\left(y''+\left(\frac{n-1}{r}+\phi'\right) y'-\frac{k}{r^2}y\right)^2 (Y_l)^2\,d\HH^{n-1}\,r^{n-1}e^{\phi(r)}\,dr.
\]
By our chosen normalizations, the angular portion of the integral simplifies to $1$, and so we are left with
\[
 N[u]=\int_0^R\left(y''+\left(\frac{n-1}{r}+\phi'\right) y'-\frac{k}{r^2}y\right)^2\,r^{n-1}e^{\phi(r)}\,dr.
\]
We next investigate the range of $k$-values for which $N[u]$ is increasing. For this investigation, we need only consider the terms involving $k$, namely
\begin{equation}\label{eqn:kstuff}
 \int_0^R\left(-\frac{2k}{r^2}yy''-\frac{2k}{r^3}\left(n-1+r\phi'\right) yy'+\frac{k^2}{r^4}y^2\right)\,r^{n-1}e^{\phi(r)}\,dr. 
\end{equation}
Using integration by parts, we find
\begin{align*}
  \int_0^R-2kyy''r^{n-3}e^{\phi(r)}\,dr&=\left.\left(-2kr^{n-3}e^{\phi(r)}yy'\right)\right|_{0}^R\\
  &\qquad+2k\int_0^R\left(\frac{(y')^2}{r^2}+\frac{yy'}{r^3}\Big(n-3+r\phi'\Big)\right)r^{n-1}e^{\phi(r)}\,dr.
\end{align*}
In order to have smoothness of $u$ at the origin, we must have $y'(0)=0$ when $l=0$ and $y(0)=0$ when $l\geq 1$, and so the boundary terms from integration by parts vanish at $r=0$. The remaining boundary term vanishes by the clamped boundary conditions $y(R)=y'(R)=0$. Thus expression \eqref{eqn:kstuff} becomes
\begin{align*}
&\int_0^R\left(\frac{k^2}{r^4}y^2-\frac{2k}{r^3}\Big(n-1+r\phi'\Big)yy'+\frac{2k(y')^2}{r^2}+\frac{2k}{r^3}\Big(n-3+r\phi'\Big)yy'\right)r^{n-1}e^{\phi(r)}\,dr\\
 &=\int_0^R\left(\frac{k^2-2k}{r^4}y^2+\frac{2k}{r^4}(y-ry')^2\right)r^{n-1}e^{\phi(r)}\,dr,
\end{align*}
by completing the square. Since all terms involving $y$ are now squared, it is clear this quantity is increasing in $k$ for all $k\geq 1$. Since $k=l(l+n-2)$, it follows that $N[u]$ is increasing in $l$ when $l\geq1$.

If $l=0$, then $k=0$ and the terms involving $k$ all vanish. If $l=1$, then $k=n-1$ and so $k^2-2k=(n-1)(n-3)$, which is nonnegative because our dimension $n\geq 3$. Therefore, $N[u]$ is increasing in $l$ for all $l\geq 0$ and dimensions $n\geq 3$. Since we assumed $u$ was a normalized eigenfunction, $N[u]$ is an eigenvalue, and so our eigenvalue will be minimized when $l=0$.
\end{proof}

\section{Confluent Hypergeometric Functions}\label{Sec:CHF}
In this section we collect some relevant properties of confluent hypergeometric  functions. They appear in the radial part of solutions to both the anti-Gauss clamped plate eigenvalue problem on a ball and the $J_{A,B}$ minimization problem. Unless stated otherwise, all properties in this section can be found in \cite[Ch. 13]{NIST}. See \cite[Ch. 13]{AShandbook} for more information.

The confluent hypergeometric  functions are a class of special functions defined as those solving Kummer's equation:
\[
 z \frac{d^2w}{dz^2}+(b-z)\frac{dw}{dz}-aw=0.
\]
The standard solutions are the Kummer functions $M(a,b,z)$ and $U(a,b,z)$. Here $M$ is the function given by the power series
\[
 M(a,b,z)=\sum_{k=0}^\infty\frac{(a)_k}{(b)_kk!}z^k,
\]
where $(a)_k=a(a+1)\cdots(a+k-1)$. The function $M$ is not defined for negative integer values of $b$, but in this paper we will only consider positive values of this parameter.

Generally, the second linearly independent solution is taken to be the function $U$ uniquely defined by the property
\[
 U(a,b,z)\sim z^{-a},\qquad z\to\infty,\qquad |\mathrm{ph}z|\leq \frac{3\pi}{2}-\delta,
\]
where $\mathrm{ph}z$ is the phase of $z$ (as a complex number). However, since we will be most interested in $z$ near the origin, we will instead take our second solution to be a function $\UT$ so that $M(a,b,z)$ and $\UT(a,b,z)$ are a so-called ``numerically satisfactory pair'' near the origin (see \cite[Sec 2.7(iv)]{NIST}) -- that is, they are linearly independent and do not exhibit a large amount of cancellation near $z=0$.

Our choice of $\UT$ depends on the values of $a$ and $b$, following \cite[Sec 13, p.323-324]{NIST}. When $b\geq0$ is a noninteger, we take
\[
 \UT(a,b,z)=z^{1-b}M(a-b+1,2-b,z).
\]
When $b$ is a positive integer, and $a-b+1\neq 0,-1,-2,\dots$, we take
\[
 \UT(a,b,z)=\sum_{k=1}^{b-1}\frac{(b-1)!(k-1)!}{(b-k-1)!(1-a)_k}z^{-k}-\sum_{k=0}^\infty\frac{(a)_k}{(b)_kk!}z^k\left(\ln(z)+\psi(a+k)-\psi(1+k)-\psi(b+k)\right),
\]
where $\psi(x)=\Gamma'(x)/\Gamma(x)$, the logarithmic derivative of the $\Gamma$ function. 

If instead $a=1,2,\dots, b-1$, then
\[
 \UT(a,b,z)=\sum_{k=a}^{b-1}\frac{(k-1)!}{(b-k-1)!(k-a)!}z^{-k}.
\]
If we have $a=0,-1,-2,\dots$, then
\begin{align*}
\UT(a,b,z)&=\sum_{k=1}^{b-1}\frac{(b-1)!(k-1)!}{(b-k-1)!(1-a)_k}z^{-k}+(-1)^{1+|a|}(|a|)!\sum_{k=1+|a|}^\infty\frac{(k-1+a)!}{(b)_kk!}z^k\\
 &\qquad-\sum_{k=0}^{|a|}\frac{(a)_k}{(b)_kk!}z^k\left(\ln(z)+\psi(1-a-k)-\psi(1+k)-\psi(b+k)\right).
\end{align*}
We will not be considering negative values for $b$ in this paper, so we do not need to address $\UT$ in this case.

Note that for all values of $a$ and $b$ under consideration, the function $M$ is entire and real-valued when $z\in\RR$. However, the function $\UT$ may be singular at the origin, and complex-valued when $z$ is a negative real. We take the principal branches of $\ln(z)$ and powers of $z$, with the usual branch cuts along the negative real axis.

\subsection*{Asymptotics and roots}\label{Sect:Asymp}
All asymptotics given below are taken as $z\to0$.

For all choices of $a\in\RR,b>0$, we have
\begin{equation*}
M(a,b,z)=1+\frac{a}{b}z+\BigO(z^2).
\end{equation*}
If $b>0$ is not an integer, 
\begin{equation}\label{UT0}
\UT(a,b,z)=z^{1-b}\left(1+\frac{b-a-1}{b-2}z+\BigO(z^2)\right).
\end{equation}
When $b$ is a positive integer, our asymptotics depend on $a$. If $b\geq2$ and $a=1,2,\dots,b-1$, then

\begin{numcases}{ \UT(a,b,z)=}
             \frac{(b-2)!}{(b-a-1)!}z^{1-b}\left(1+\frac{b-a-1}{b-2}z+\BigO(z^{2})\right), & $b-a\geq 3$, \label{UT1}\\
               (b-2)!z^{1-b}+(b-3)!z^{2-b}, &$b-a=2$,\label{UT2}\\
               (b-2)!z^{1-b}, &$b-a=1$. \label{UT3}
               \end{numcases}
For $b\geq 2$ and all other values of $a$, we have
\begin{numcases}{\UT(a,b,z)=}
\frac{(b-1)!(b-2)!}{(1-a)_{b-1}}z^{1-b}\left(1+\frac{b-a-1}{b-2}z+\BigO(z^{2})\right), &$b\geq 4$,\label{UT4}\\
 \frac{2}{(1-a)(2-a)}z^{-2}+\frac{2}{1-a}z^{-1}+\BigO(\ln(z)), &$b=3$,\label{UT5}\\
\frac{1}{(1-a)}z^{-1}-\ln(z)+\BigO(1), &$b=2$.\label{UT6}
\end{numcases}
If $b=1$, then
\begin{numcases}{\UT(a,b,z)=}
   -\ln(z)-\psi(1-a)+2\psi(1)-a z\ln(z)+\BigO(z), &$a=0,-1,-2,\dots,$ \label{UT7}\\
   -\ln(z)-\psi(a)+2\psi(1)-a z\ln(z)+\BigO(z), &\text{otherwise}. \label{UT8}
\end{numcases}

We will also need properties of roots of $M(a,b,z)$. Let $p(a,b)$ denote the number of positive real zeros of $M(a,b,z)$ and $n(a,b)$ denote the number of negative real zeros of $M(a,b,z)$. 
Then for all $b\geq0$ and $a\in\RR$, 
\begin{align*}
 p(a,b)&=\begin{cases}\lceil|a|\rceil &\text{if $a< 0$,}\\
          0 &\text{if $a\geq0$,}
         \end{cases}
         \end{align*}
and  
         \begin{align}\label{NegZeroes}
 n(a,b)&=p(b-a,b)=\begin{cases}\lceil a-b\rceil &\text{if $b<a$,}\\
          0 &\text{if $b\geq a$.}
         \end{cases}
             \end{align}
Note in particular that when $b\geq a$, the function $M(a,b,z)$ has no negative zeros.

\section{Eigenfunctions of the ball in anti-Gauss space}\label{Section:Ball}
In this section, we shall find the eigenfunctions of $\AF^2$ on balls using a separation of variables approach.

\begin{prop}\label{ballspcprop} Let $a(x)=e^{|x|^2/2}$ denote an anti-Gaussian weight. If $R>0$ and $\BB_R=\{x\in \RR^n:|x|<R\}$, then the solutions to the eigenvalue problem
\begin{equation}\label{Eqn:PDEProbBall}
\begin{cases}
  \AF^2u=\Lambda u &\text{in $\BB_R$,}\\
  u=\frac{\partial u}{\partial r}=0 &\text{when $|x|=R$,}
 \end{cases}
 \end{equation}
take the form
\begin{equation}\label{balleigenfcn}
 u(r,\that)=A Y_l(\that)\left(M\left(\frac{l+\lambda}{2},\frac{n}{2}+l,-\frac{r^2}{2}\right)+G_R M\left(\frac{l-\lambda}{2},\frac{n}{2}+l,-\frac{r^2}{2}\right)\right)
\end{equation}
with corresponding eigenvalues
\[
 \Lambda=\lambda^2.
\]
Here the functions $Y_l(\that)$ are spherical harmonics of order $l$, the constant $G_R$ is chosen so that $u(R,\that)\equiv 0$, and the constant $\lambda$ is taken to satisfy the boundary condition $\partial u/\partial r=0$ at $r=R$.
\end{prop}

\begin{remark} When dimension $n$ and order $l$ are fixed, we will use the shorthand notation
\[
 M_\pm(z)=M\left(\frac{l\pm \lambda}{2},\frac{n}{2}+l,z\right)\qquad\text{and}\qquad M_\pm'(z)= M'\left(\frac{l\pm\lambda}{2},\frac{n}{2}+l,z\right).
\]
\end{remark}

\begin{proof} 
By Proposition~\ref{Thm:RadMod}, eigenfunctions of the ball take the form $u=y(r)Y_l(\hat\theta)$, where $Y_l$ is a spherical harmonic of order $l$, and $y$ satisfies the ODE problem
\begin{equation}\label{Eqn:ODEy}
 \begin{cases}\left(\frac{\partial^2}{\partial r^2}+\left(\frac{n-1}{r}+r\right)\frac{\partial}{\partial r}-\frac{l(l+n-2)}{r^2}\right)^2y=\Lambda y \qquad\text{when $r\in(0,R)$,}\\
  y(R)=y'(R)=0.
 \end{cases}
\end{equation}
We will show that $y$ has the form of \eqref{balleigenfcn}.

For the sake of convenience, we write
\[
\AF_r=\frac{\partial^2}{\partial r^2}+\left(\frac{n-1}{r}+r\right)\frac{\partial}{\partial r}-\frac{l(l+n-2)}{r^2},
\]
so that the ODE above can be written as $\AF_r^2y=\Lambda y$. Writing $\Lambda=\lambda^2$, we factor this ODE as 
\[
 (\AF_r-\lambda)(\AF_r+\lambda)y=0.
\]
The solutions will then be linear combinations of the solutions of $\AF_ry=\pm\lambda y$, that is,
\begin{equation}\label{factoredpm}
 y''+\left(\frac{n-1}{r}+r\right)y'-\frac{l(l+n-2)}{r^2}y=\pm\lambda y.
\end{equation}
If we make the substitutions $y(r)=r^lw(-r^2/2)$ and $z=-r^2/2$, the above equation can be transformed into the Kummer differential equation
\[
 z\frac{d^2w}{dz^2}+\left(\frac{n+2l}{2}-z\right)\frac{dw}{dz}-\frac{l\pm\lambda}{2}w=0.
\]
The solutions to these equations are the confluent hypergeometric  Kummer functions $M((l\pm\lambda)/2,n/2+l,z)$ and $\UT((l\pm\lambda)/2,n/2+l,z)$. We thus have that the solutions to \eqref{factoredpm} have the general form
\[
 y(r)=r^l\left(C_1 M\left(\frac{l\pm\lambda}{2},\frac{n}{2}+l,-\frac{r^2}{2}\right)+C_2\UT\left(\frac{l\pm\lambda}{2},\frac{n}{2}+l,-\frac{r^2}{2}\right)\right),
\]
and so the solutions to $\AF^2y=\Lambda y$ are
\begin{align}\label{Eqn:ydecomp}
 y(r)&=r^l\Big(C_1 M(l/2+\lambda/2,n/2+l,-r^2/2)+ C_2M(l/2-\lambda/2,n/2+l,-r^2/2) \\
 &\qquad+C_3\UT(l/2+\lambda/2,n/2+l,-r^2/2)+ C_4\UT(l/2-\lambda/2,n/2+l,-r^2/2)\Big),\nonumber
\end{align}
where $C_1,C_2,C_3,C_4$ are (possibly complex-valued) constants.

Recall from our regularity result in Proposition~\ref{spect} that eigenfunctions are smooth on $\BB_R$. However, the Kummer functions $\UT$ are potentially singular at the origin, while the $M$ are not. Furthermore, because the eigenfunctions are smooth and $Y_l(\that)$ changes sign for $l\geq 1$, we must have $y(0)=0$. When $l=0$, the spherical harmonic $Y_0$ is constant and we require the radial function satisfy $y'(0)=0$ to ensure smoothness. A delicate case by case analysis using this information shows that the Kummer functions $\UT$ must be removed from the linear combination \eqref{Eqn:ydecomp}; details are carefully presented in the Appendix.

We next impose the clamped boundary conditions $y(R)=y'(R)=0$. Since $l/2-\lambda/2-(n/2+l)=-(\lambda+n+l)/2$ is negative, by \eqref{NegZeroes} the function $\Mm(-r^2/2)=M(l/2-\lambda/2,n/2+l,-r^2/2)$ has no roots for $r>0$. Thus $C_1\neq 0$ and the condition $y(R)=0$ can be rewritten as
\begin{equation}\label{Eqn:DefOfGR}
\frac{C_2}{C_1}=\frac{-\Mp(-R^2/2)}{\Mm(-R^2/2)}=:G_R.
\end{equation}
By rescaling, we assume
\[
y(r)=r^l \Big( \Mp(-r^2/2)+G_R\Mm(-r^2/2)\Big ).
\]
Imposing the final boundary condition $y'(R)=0$ allows us to determine $\lambda$. Note that
\[
y'(r)=lr^{l-1}\Big(\Mp(-r^2/2)+G_R\Mm(-r^2/2)\Big)-r^{l+1}\Big(\Mpd(-r^2/2)+G_R\Mmd(-r^2/2)\Big),
\]
and so
\[
0=y'(R)=\frac{l}{R}y(R)-R^{l+1}\Big(\Mpd(-R^2/2)+G_R\Mmd(-R^2/2)\Big),
\]
or equivalently
\begin{equation}\label{Eqn:DefOfh}
0=\Mpd(-R^2/2)\Mm(-R^2/2)-\Mmd(-R^2/2)\Mp(-R^2/2)=:h_R(\lambda).
\end{equation}
\end{proof}

\begin{remark}
We can use the function $h_R(\lambda)$ to numerically compute the eigenvalues of the ball $\BB_R$. If $\lambda$ is a root of $h_R(\lambda)$ for a given $R$, then $\Lambda(\BB_R)=\lambda^2$ is an eigenvalue. This makes it quite easy to numerically verify that the conclusion of Proposition~\ref{Thm:RadMod} still holds true for dimension $n=2$ in the anti-Gaussian setting. Figure~\ref{fig:n=0,l=0,1} shows that when the solution curves of $h_R(\lambda)=0$ are plotted in the $(R,\lambda)$ plane, for each $R$, the smallest $\lambda$ corresponding to $l=0$ is less than that for $l=1$. 

\begin{figure}[h!]
 \begin{center}
  \includegraphics[width=2in]{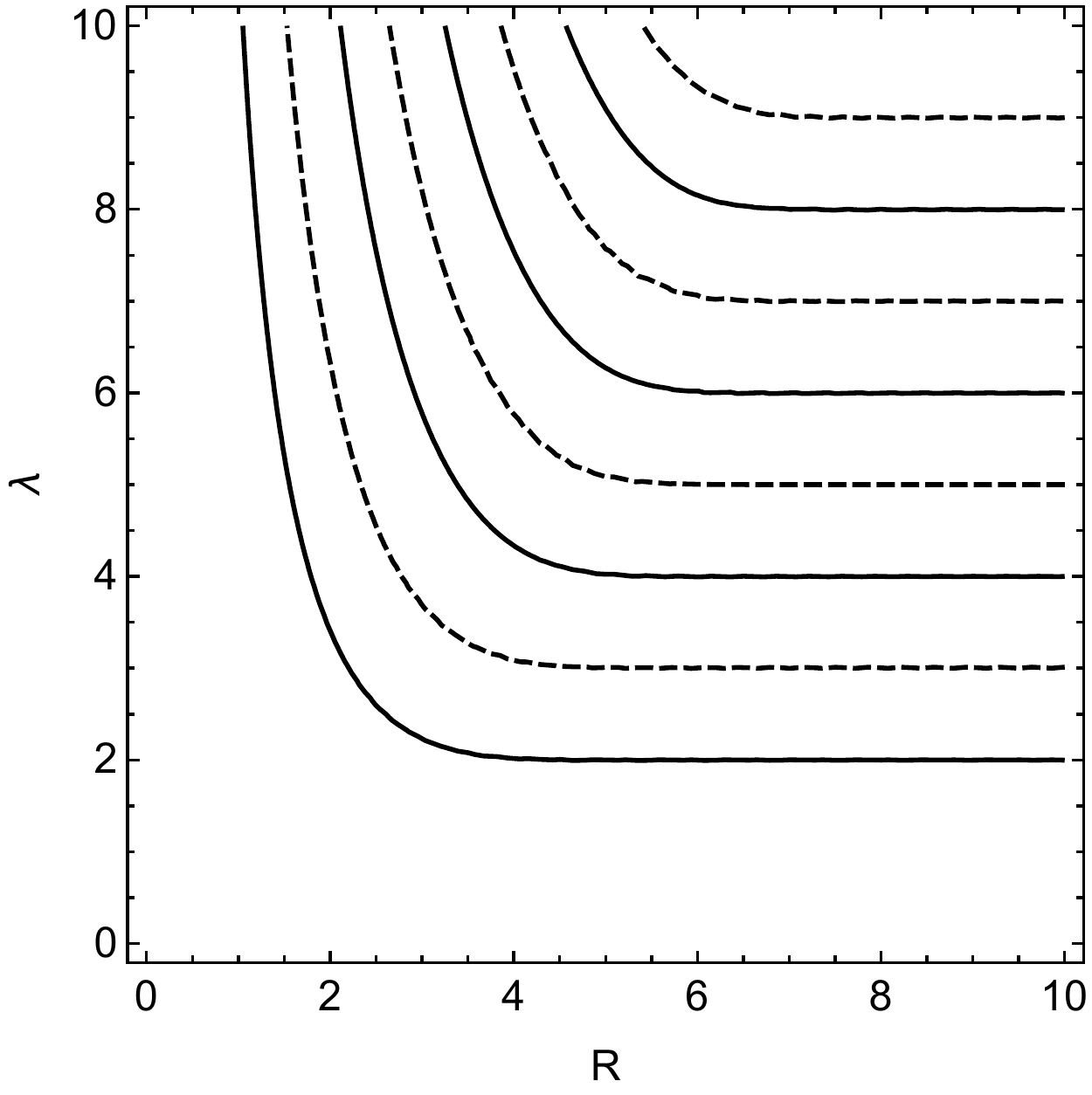}
  \caption{\label{fig:n=0,l=0,1} The solution curves of $h_R(\lambda)=0$ for $l=0$ (solid) and $l=1$ (dashed).}
 \end{center}
\end{figure}
 
\end{remark}

\section{The Euler-Lagrange system.}\label{Section:ELSoln}
We now revisit the $J_{A,B}$ minimization problem of Section \ref{Sect:ProofMainRes}, considering the anti-Gaussian weight $a(x)=e^{|x|^2/2}$. Recall the notation $\BB_A=\{x\in\RR^n:|x|<A\}$ and
\begin{equation}\label{Jabprob}
 J_{A,B}=\inf_{(v,w)} \ \frac{\int_{\BB_A}(\AF v)^2\,d\alpha+\int_{\BB_B}(\AF w)^2\,d\alpha}{\int_{\BB_A}v^2\,d\alpha+\int_{\BB_B}w^2\,d\alpha},
 \end{equation}
where the $\inf$ ranges over all pairs of radial functions $(v,w)$ satisfying
\begin{align*}
&v\in H^2(\BB_A,\alpha)\cap H^1_0(\BB_A,\alpha),\quad w\in H^2(\BB_B,\alpha)\cap H^1_0(\BB_B,\alpha),\\
&A^{n-1}e^{A^{2}/2}\frac{\partial v}{\partial r}(A)=B^{n-1}e^{B^{2}/2}\frac{\partial w}{\partial r}(B).
 \end{align*}

Recall that in Remark \ref{Rmk:JABpos} we showed $J_{A,B}>0$. Also, for each pair $(v,w)$, the single-variable dependence allows us to write $\AF v$ and $\AF w$ in terms of partial derivatives with respect to $r$. Thus we will simply write $\AF v = v''+((n-1)/r+r)v'$ throughout this section.

A standard calculus of variations argument shows that a minimizing pair $(v,w)$ of $J_{A,B}$ satisfies the following Euler-Lagrange system:
\begin{numcases}{}
 &$\AF^2 v=\mu v \qquad \textup{on}\quad \{r<A\}$,\nonumber\\
 &$\AF^2 w=\mu w \qquad \textup{on}\quad \{r<B\}$,\nonumber\\
 &$v(A)=w(B)=0$,\label{ELfixed}\\
 &$A^{n-1}e^{A^2/2}v'(A)=B^{n-1}e^{B^2/2}w'(B),$\label{ELderiv}\\
 &$\AF v(A)+\AF w(B)=0$.\label{ELAcond}
\end{numcases}
The smallest eigenvalue $\mu$ satisfying this system is $J_{A,B}$. The solutions to the system can be expressed in terms of confluent hypergeometric functions; relevant properties of these functions appear in Section~\ref{Sec:CHF}.

\subsection*{Solving the Euler-Lagrange system}
From our work in Section~\ref{Section:Ball}, we already know that radial solutions to $\AF^2 v=\mu v$ on the ball $\BB_A$ take the form
\[
 v(r)=C_1\Mp(-r^2/2)+ C_2\Mm(-r^2/2)+C_3{\UT_{+}}(-r^2/2)+C_4{\UT_{-}}(-r^2/2),
\]
where we recall the shorthand notation
\[
M_\pm(z)=M\left(\pm\lambda/2,n/2,z\right)\qquad\text{and}\qquad \UT_\pm(z)=\UT \left(\pm\lambda/2,n/2,z\right),
\]
where $\lambda=\sqrt{\mu}$. As we remarked in the proof of Proposition \ref{Prop:MinExt}, minimizers $(v,w)$ for $J_{A,B}$ are smooth on their respective domains. Using the series expansions for $M_\pm$ and $\UT_\pm$, our work in the Appendix (considering only cases where $l=0$) shows that $C_3$ and $C_4$ must vanish to ensure regularity of $v$. Analogous considerations apply to $w$.

We next impose the boundary conditions of our Euler-Lagrange system. Recall that in \eqref{Eqn:DefOfGR} we defined 
\[
G_R=\frac{-\Mp(-R^2/2)}{\Mm(-R^2/2)};
\]
the definition of $G_R$ came by requiring $y(R)=0$ in \eqref{Eqn:ODEy}. The Euler-Lagrange system includes the analogous condition \eqref{ELfixed}, and so we take $v,w$ to be
\begin{align*}
 v(r)&=C_v\Big(\Mp(-r^2/2)+G_A\Mm(-r^2/2)\Big),\\
 w(r)&=C_w\Big(\Mp(-r^2/2)+G_B\Mm(-r^2/2)\Big),
\end{align*}
where $C_v,C_w$ are constants which are not both zero. We impose the remaining Euler-Lagrange conditions to find a relationship between these constants and hence a condition on the eigenvalue $\mu$. Using the differential equations that $M_\pm$ satisfy together with the definition of $G_R$, we see that condition \eqref{ELAcond} becomes
\begin{equation}\label{Eqn:MatrixCoeff1}
 0=2\lambda\Big(C_v\Mp(-A^2/2)+C_w\Mp(-B^2/2)\Big).
\end{equation}
Our final boundary condition is \eqref{ELderiv}:
\begin{equation}\label{Eqn:MatrixCoeff2}
 A^{n}e^{A^2/2}C_v\Big(\Mpd(-A^2/2)+G_A\Mmd(-A^2/2)\Big)=B^{n}e^{B^2/2}C_w\Big(\Mpd(-B^2/2)+G_B\Mmd(-B^2/2)\Big).
\end{equation}
Viewing \eqref{Eqn:MatrixCoeff1} and \eqref{Eqn:MatrixCoeff2} as a system of equations, we must have a nontrivial solution in $C_v,C_w$. The corresponding determinate must vanish, and using the definitions of $G_R$ and $h_{\lambda}$ in \eqref{Eqn:DefOfh}, a bit of algebra leads to the condition
\begin{equation}
A^ne^{A^2/2} h_A(\lambda)\Mp(-B^2/2)\Mm(-B^2/2)+B^ne^{B^2/2}h_B(\lambda)\Mp(-A^2/2)\Mm(-A^2/2)=0. \label{Eqn:JABequation} 
\end{equation}
This equation determines the value of $\mu$ in terms of $A$, $B$, and $n$. It is this condition that we use to numerically compute values of $\mu=J_{A,B}$.

\section{Numerics of the anti-Gaussian constants $C(R,n)$}\label{Sect:NumLowBds}

Our numerical values for the constants $C(R,n)$ in anti-Gauss space are plotted in Figure~\ref{fig:const} for low dimensions. Here, we briefly summarize our methodology for computing these values. For each dimension $n$ and volume $V$, the constant $C(V,n)$ is defined as in \eqref{eqn:Cdef}, as the ratio of the minimal $J_{A,B}$ over $\Lambda_1(\BB_R)$. The constant $C(R,n)$ is obtained by taking $V=\alpha(\BB_R)$. We compute the eigenvalues $\Lambda_1(\BB_R)$ by numerically solving the equation $h_R(\lambda)=0$, where $h$ is defined as in \eqref{Eqn:DefOfh}; the smallest positive solution $\lambda$ of this equation equals $\sqrt{\Lambda_1(\BB_R)}$. We estimate $J_{A,B}$ numerically by solving \eqref{Eqn:JABequation} for a each pair of numbers $(A,B)$, and we minimize over such pairs $(A,B)$. By symmetry, we need only consider the values of $A$ ranging from $0$ to $A^*(R)$, where $A^*(R)$ satisfies $2\alpha(\BB_{A^*(R)})=\alpha(\BB_R)$, i.e., when $A=B$. 

Our numerics show that for all dimensions $n\geq 2$, the constants $C(R,n)\geq 0.85$ for all radii $R$, with $C(R,n)\to1$ as $R\to\infty$. We also see that (up to machine rounding error), the constant $C(R,2)=1$ for $0<R\leq1.23$ and $C(R,3)=1$ for $0<R\leq0.5$.
For such values of $R$, we see that $J_{A,B}$ is minimized when all of the mass is ``pushed'' into a single ball. For larger values of $R$, the minimizing $(A,B)$ pair appears to occur either at or very close to $A=B$, so that $J_{A,B}$ is minimized when mass is evenly distributed across two balls of equal radius. These phenomena are illustrated in Figure~\ref{fig:Jabmin23}. The transition between the two regimes occurs abruptly at $R\approx 1.24$ for dimension $n=2$ and $R\approx0.6$ for dimension $n=3$.

For higher dimensions and all radii, the minimizing $(A,B)$ pair also appears to occur at or near two balls of equal radius. Taken in sum, our observations curiously mirror the Euclidean case. There, in dimensions $n=2$ and $3$, Ashbaugh and Benguria \cite{AB95} showed that the Euclidean $J_{A,B}$ is minimized when a single ball has all the mass, and in dimensions $n\geq 4$, Ashbaugh and Laugesen \cite{AL96} showed that the Euclidean $J_{A,B}$ is minimized by two balls of equal radius.

\begin{figure}
 \includegraphics[width=3in]{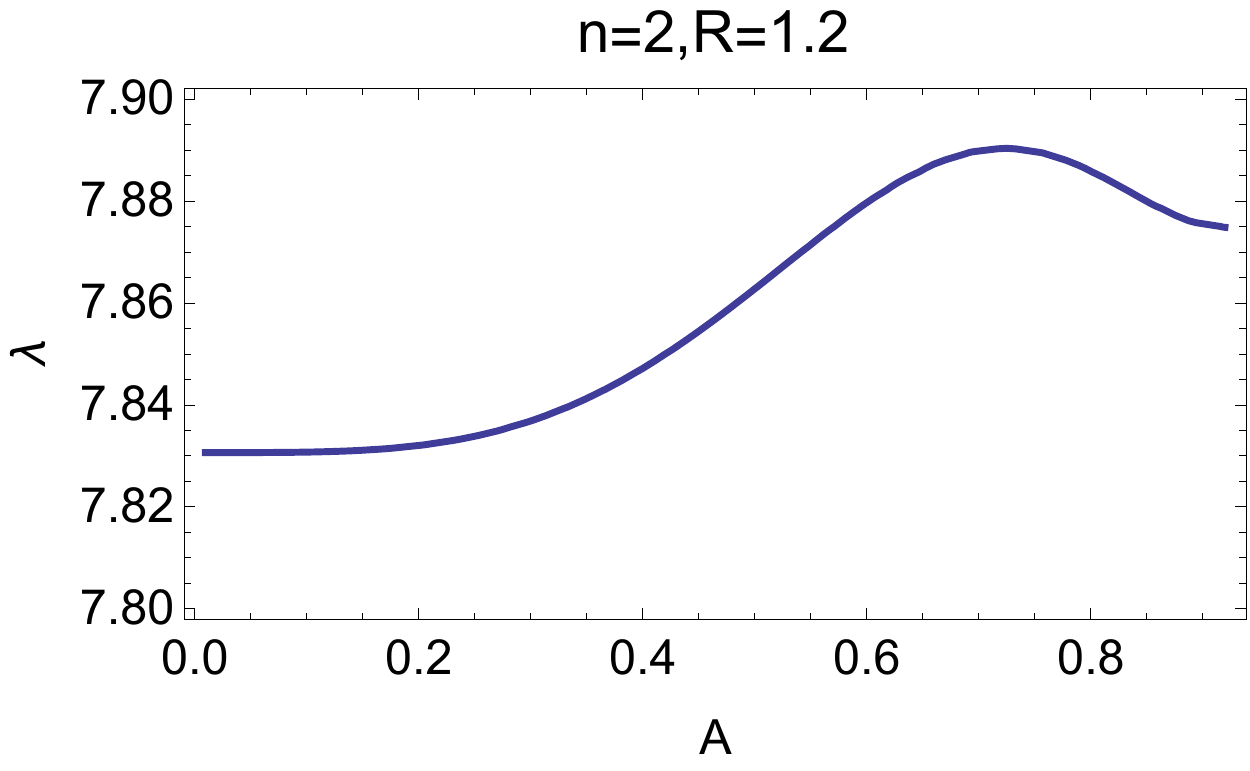}\qquad\includegraphics[width=3in]{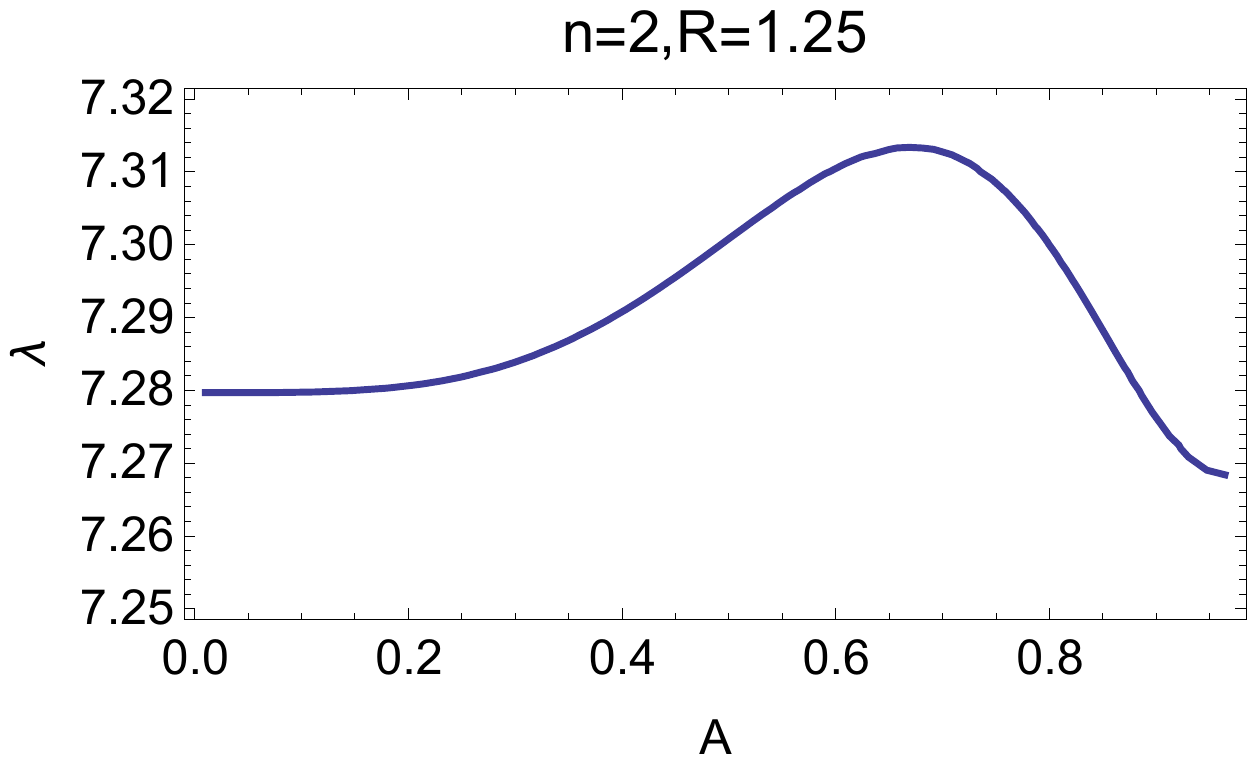}\\
 \includegraphics[width=3in]{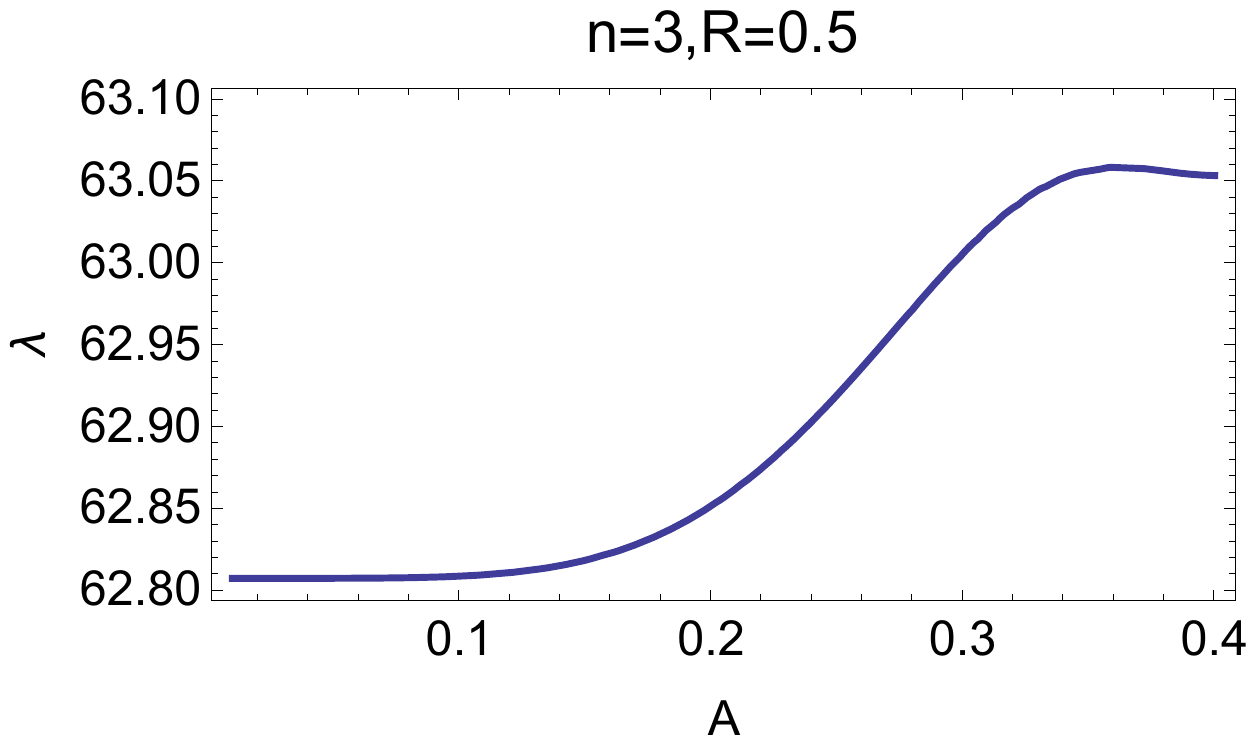}\qquad\includegraphics[width=3in]{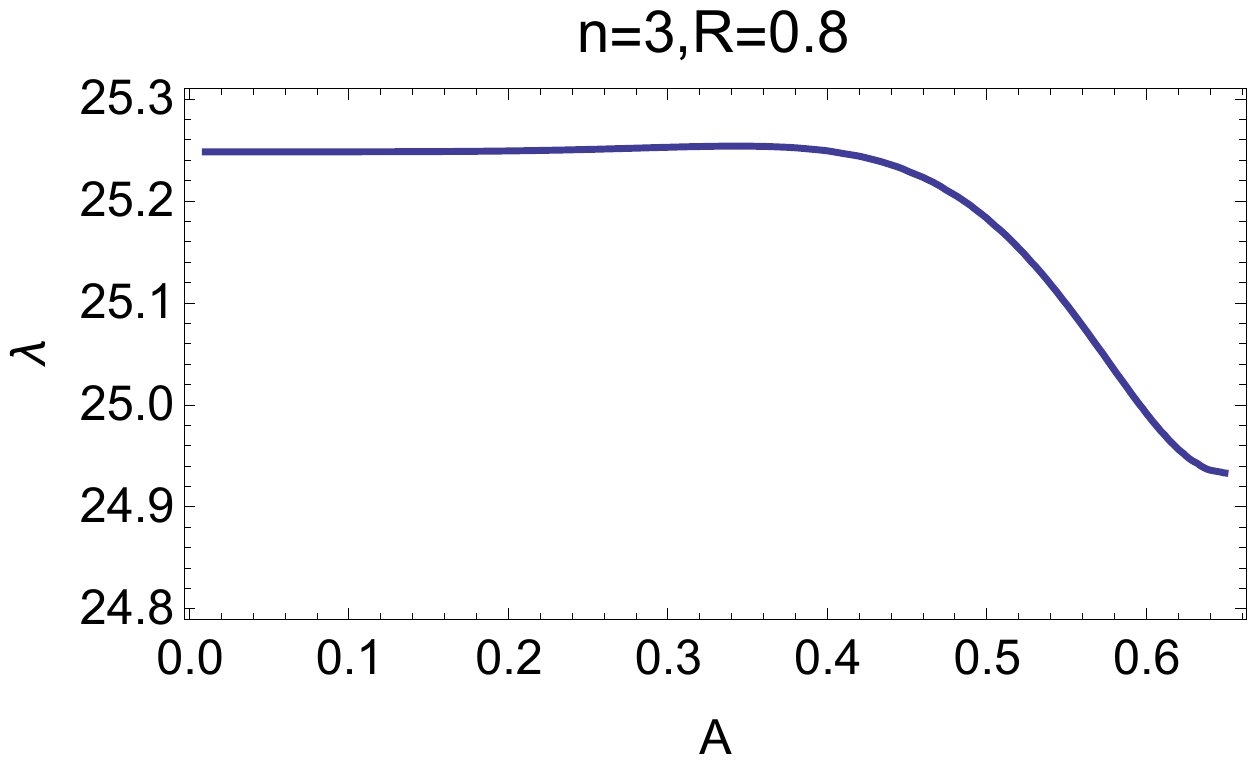}
 \caption{\label{fig:Jabmin23} Graphs of $\lambda=\sqrt{J_{A,B}}$ as a function of $A$ for $R=1.2$ and $R=1.25$ in dimension $n=2$, and $R=0.5$ and $R=0.8$ in dimension $n=3$. The parameter $A$ ranges from zero to $A^*(R)$. Note the location of the minimizer changes from $A\approx 0$ (all mass on one ball) to $A\approx A^*(R)$ (mass distributed evenly across two balls).}
\end{figure}

\begin{figure}
 \includegraphics[width=3in]{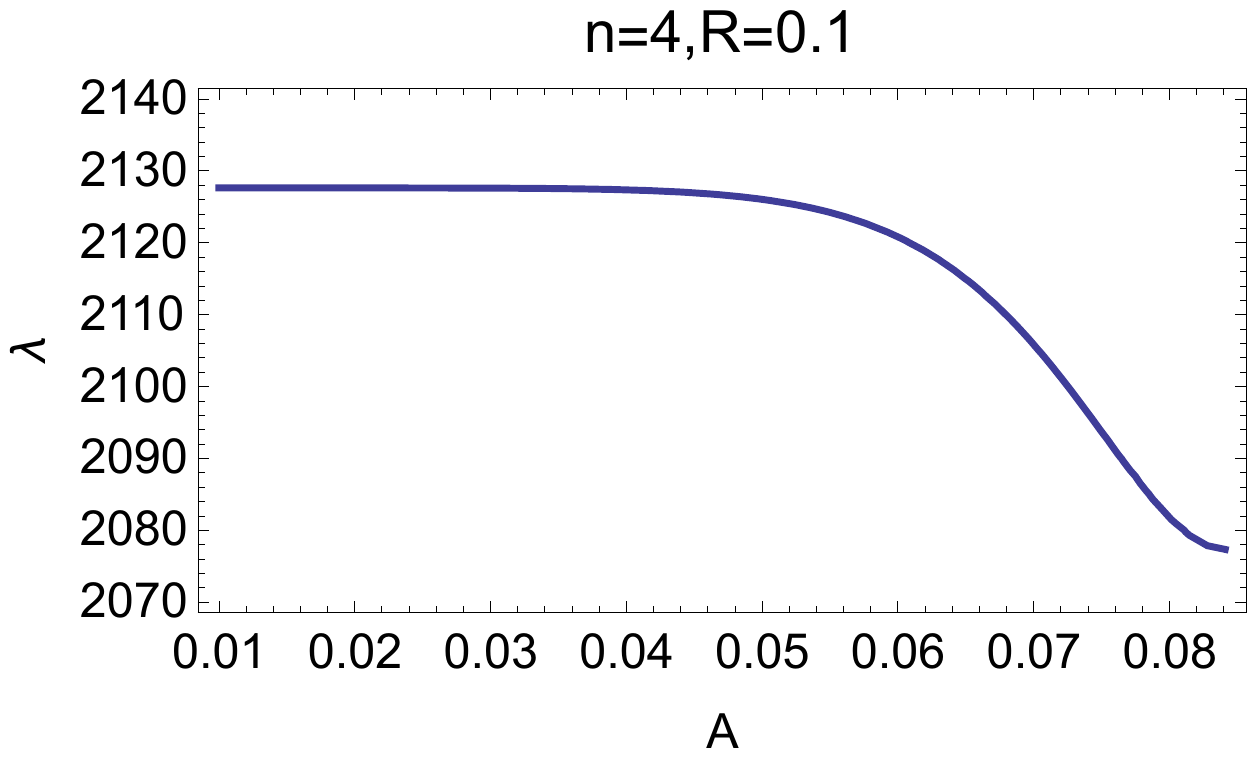}\qquad\includegraphics[width=3in]{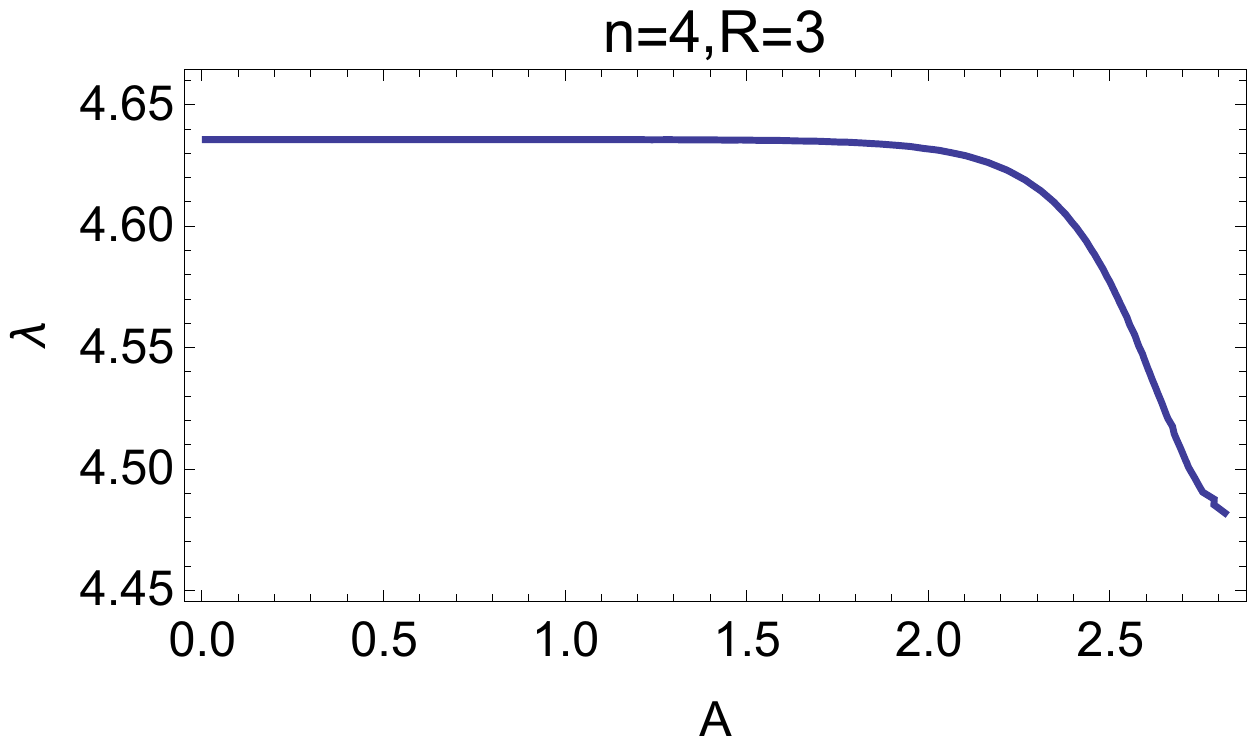}
 \caption{\label{fig:Jabmin4} Graphs of $\lambda=\sqrt{J_{A,B}}$ as a function of $A$ for $R=0.1$ and $R=3$ for dimension $n=4$. Note that the location of the minimizer occurs at  $A\approx A^*(R)$ (mass distributed evenly across two balls).}
\end{figure}

\section*{Appendix}
For the purposes of this appendix, we return our attention to Proposition \ref{ballspcprop}. We saw in the proof of this result that solutions to  \eqref{Eqn:PDEProbBall} separate as
\[
u(r,\hat \theta)=y(r)Y_l(\hat \theta),
\]
and that $y(r)$ decomposes as in \eqref{Eqn:ydecomp}. The goal of this appendix is to show that $y(r)$ consists only of linear combinations of Kummer functions $M_\pm$ and not $\UT_\pm$. Our arguments hinge on two simple observations. First, when viewed as a function on $\BB_R$, $y(r)$ is smooth. In particular, $y(r)$ is smooth at $r=0$. Second, $y'(0)=0$ whenever the parameter $l=0$, while $y(0)=0$ whenever the parameter $l\geq1$.

The work of our appendix also applies (by considering only cases with $l=0$) to the minimizing pair $(v,w)$ of Proposition \ref{Prop:MinExt}, and shows that these functions likewise exhibit a decomposition in terms of $M_\pm$.

We shall make use of the following shorthand notation:
\[
\UT_\pm(z)=\UT \left(\frac{l\pm\lambda}{2},\frac{n}{2}+l,z\right),\qquad a_\pm=\frac{l\pm\lambda}{2},
\]
and we frequently refer to the linear combinations
\begin{equation}\label{Eqn:MPMlincomb}
 r^l\Big(C_1\Mp(-r^2/2)+C_2\Mm(-r^2/2)\Big),
\end{equation}
\begin{equation}\label{Eqn:UPMlincomb}
r^l\left(C_3\UT_+(-r^2/2)+C_4\UT_-(-r^2/2)\right),
\end{equation}
consisting of the first and final two terms of \eqref{Eqn:ydecomp}. Our work heavily relies on the asymptotic formulas of Section \ref{Sect:Asymp}. We divide our work by the parity of the dimension.

Note that since our eigenvalues $\Lambda$ and $J_{A,B}$ are all positive, we have $\lambda>0$ and hence $a_+>a_-$ regardless of dimension and value of $l$.

\subsection*{Odd dimensions}
Since our dimension $n\geq3$ is odd, we have that $b=n/2+l$ is a noninteger. Regardless of the values of $a_\pm$, the functions $\UT_\pm$ obey \eqref{UT0}. To eliminate $r^{2-n-l}$ from \eqref{Eqn:UPMlincomb}, we  take $C_4=-C_3$ and \eqref{Eqn:UPMlincomb} becomes
\begin{equation}\label{Appnodd}
C_3r^{4-n-l}\left(-\frac{1}{2}\right)^{1-b}\left(\frac{\lambda}{2b-4}+\BigO(r^2)\right).
\end{equation}
This final expression remains singular at $r=0$ if $4-n-l<0$ unless $C_3=0$. This condition on $n$ and $l$ is met when $n\geq5$ for all $l\geq0$ or $n=3$ for $l\geq2$. 

If $n=3$ and $l=1$, then $r^{4-n-l}=1$ and $b=5/2$, and no singularities remain in \eqref{Appnodd}. However, in this case, smoothness of $y(r)$ also requires $y(0)=0$. We note that \eqref{Eqn:UPMlincomb} becomes
\[
C_3\left(-\frac{1}{2}\right)^{-3/2}\left(\lambda+\BigO(r^2)\right),
\]
while \eqref{Eqn:MPMlincomb} is $\BigO(r)$, and so to ensure $y(0)=0$, we must have $C_3=0$.

If $n=3$ and $l=0$, then $r^{4-n-l}=r$ and $b=3/2$, and no singularities remain in \eqref{Appnodd}; this time, smoothness additionally requires $y'(0)=0$. However, we have linearity in the $\UT_\pm$ contributions to \eqref{Eqn:UPMlincomb}, which becomes
\[
C_3r\left(-\frac{1}{2}\right)^{-1/2}\left(-\lambda+\BigO(r^2)\right),
\]
while the $M_\pm$ terms contribute only even powers of $r$ to \eqref{Eqn:MPMlincomb}, which becomes
\[
C_1+C_2+\frac{2}{3}\Big(C_1a_++C_2a_-\Big)r^2+\BigO(r^4).
\]
Thus, regardless of $C_1$ and $C_2$, $y'(0)=0$ forces $C_3=0$.

\subsection*{Even dimensions}
For even dimensions $n\geq 2$, the parameter $b=n/2+l\geq 1$ is an integer, so our choice of $\UT$ depends on $a$. Note that $a_++a_-=l$ is an integer, so $a_+$ is an integer if and only if $a_-$ is an integer. We first assume $a_\pm$ are nonintegers; the form of the asymptotics of $\UT_\pm$ now depend on $b$.

If $b=1$, then $n=2$ and $l=0$ so we have $a_\pm=\pm\lambda/2$. Since $a_\pm$ are nonintegers, both $\UT_\pm$ obey \eqref{UT8}. Thus to eliminate $r^l\ln(-r^2/2)$ from \eqref{Eqn:UPMlincomb}, we must choose $C_4=-C_3$, and \eqref{Eqn:UPMlincomb} becomes
\[
C_3\left(\phi(a_-)-\phi(a_+)+\frac{\lambda}{2}r^2\ln(-r^2/2)+\BigO(r^2)\right).
\] 
In order for the function above to be smooth at $r=0$, we require $C_3=0$.

 If $b=2$, then $l\leq 1$ and both $\UT_\pm$ obey \eqref{UT6}. To eliminate $r^{l-2}$ from \eqref{Eqn:UPMlincomb}, we take $C_4=C_3\frac{1-a_-}{1-a_+}$, and \eqref{Eqn:UPMlincomb} becomes
\[
C_3\left(\left(\frac{1-a_-}{1-a_+}-1\right)r^{l}\ln(-r^2/2)+\BigO(r^l)\right).
\]
In order to preserve smoothness at $r=0$, we must choose $C_3=0$.

 If $b=3$, then $l \leq 2$ and $\UT_\pm$ obey \eqref{UT5}. To eliminate $r^{l-4}$ from \eqref{Eqn:UPMlincomb}, we take $C_4=-C_3\frac{(1-a_-)(2-a_-)}{(1-a_+)(2-a_+)}$, and \eqref{Eqn:UPMlincomb} becomes
\[
C_3\left(\frac{4\lambda}{(1-a_+)(2-a_+)}r^{l-2}+\BigO(r^l\ln(-r^2/2))\right).
\]
The function above remains singular at $r=0$ if $l<2$. If $l=2$, then \eqref{Eqn:MPMlincomb} contributes $r^2$ terms and higher to \eqref{Eqn:ydecomp}. Thus in order to ensure $y(0)=0$ we must have $C_3=0$.

If $b\geq 4$, then the asymptotics of $\UT_\pm$ obey \eqref{UT4}, and to eliminate $r^{2-2b+l}$ from \eqref{Eqn:UPMlincomb}, we take $C_4=-C_3\frac{(1-a_-)_{b-1}}{(1-a_+)_{b-1}}$, and so \eqref{Eqn:UPMlincomb} becomes
\[
C_3\frac{(b-1)!(b-2)!}{(1-a_+)_{b-1}}\left(-\frac{1}{2}\right)^{1-b}r^{4-2b+l}\left(\frac{\lambda}{2b-4}+\BigO(r^2)\right).
\]
Since $b\geq 4$, we have that $4-2b+l< 0$ and we must take $C_3=0$ to avoid a singularity at $r=0$.

We assume for the remainder of the appendix that $a_\pm \in \mathbb{Z}$. Our work is divided into cases based upon the value of $b$. 

If $b=1$, then since $b=n/2+l$, it must be that $n=2$, $l=0$, and $a_\pm=\pm\lambda/2$. Since $a_+>0$, $\UT_+$ obeys \eqref{UT8} while $\UT_{-}$ obeys \eqref{UT7}. In order to eliminate $\ln(-r^2/2)$ from \eqref{Eqn:UPMlincomb}, we must take $C_4=-C_3$, and so \eqref{Eqn:UPMlincomb} becomes
\[
C_3r^l\left(\phi(1-a_-)-\phi(a_+)+\frac{\lambda}{2} r^2 \ln(-r^2/2)+\BigO(r^2)\right),
\]
which requires $C_3=0$ to ensure smoothness at the origin.

 If $b=2$, then since $b=n/2+l$ and $n\geq2$, we must have $l\leq 1$. There are two cases.

\begin{itemize}
\item{Case 2.1: $a_+=1$.} In this case, $\UT_+$ obeys \eqref{UT3} while $\UT_-$ obeys $\eqref{UT6}$ since $a_-<a_+$. In order to eliminate $r^{l-2}$ from \eqref{Eqn:UPMlincomb}, we must take $C_4=-(1-a_-)C_3$, and so \eqref{Eqn:UPMlincomb} becomes
\[
C_3r^l(1-a_-)\left(\ln(-r^2/2)+\BigO(1)\right).
\]
In order to preserve smoothness at $r=0$, we are forced to choose $C_3=0$.
\item {Case 2.2: $a_+\neq 1$.} In this case, $\UT_+$ obeys \eqref{UT6}. We have two subcases.
\begin{itemize}
\item {Subcase 2.2.1: $a_-=1$.} In this subcase, $\UT_-$ obeys \eqref{UT3} and the argument proceeds  in a manner that parallels the $a_+=1$ case.
\item{Subcase 2.2.2: $a_-\neq1$.} In this subcase, $\UT_-$ obeys \eqref{UT6}, and in order to eliminate $r^{l-2}$ from \eqref{Eqn:UPMlincomb} we must take $C_4=-C_3$, which makes the linear combination identically zero.
\end{itemize}
\end{itemize}

If $b=3$, then since $b=n/2+l$, we must have $l\leq 2$. We again proceed by cases.
\begin{itemize}
\item{Case 3.1: $a_+=1$ {\normalfont or} $a_+=2$.} Here $\UT_+$ obeys either $\eqref{UT2}$ or $\eqref{UT3}$, depending on the value of $b-a_+$. Since $a_-<a_+$, we have two subcases.
\begin{itemize}
\item{Subcase 3.1.1: $a_-=1$ and $a_+=2$.}
In this subcase, $\UT_-$ is given by $\eqref{UT2}$ and $\UT_+$ is given by $\eqref{UT3}$. In order to eliminate $r^{l-4}$ from \eqref{Eqn:UPMlincomb}, we must take $C_4=-C_3$, and the linear combination \eqref{Eqn:UPMlincomb} becomes $-2C_3r^{l-2}$, and so we are forced to choose $C_3=0$ to preserve smoothness at the origin or to ensure $y(0)=0$.

\item{Subcase 3.1.2: $a_-<1$.} Here, $\UT_-$ obeys \eqref{UT5}. The linear combination \eqref{Eqn:UPMlincomb} must eliminate the $r^{l-4}$ term, and so we must choose $C_4=-C_3(1-a_-)(2-a_-)/2$. If $b-a_+=1$, then \eqref{Eqn:UPMlincomb} simplifies to
\[
C_4r^l\left(-\frac{4}{1-a_-}r^{-2}+\BigO(\ln(-r^2/2))\right).
\]
In order to be nonsingular at $r=0$ (when $l\leq 1$) or to ensure $y(0)=0$ (when $l=2$), we must take $C_4=0$. If $b-a_+=2$, then \eqref{Eqn:UPMlincomb} simplifies to
\[
C_3r^l\left( 2(1-a_-)r^{-2}+\BigO(\ln(-r^2/2))\right).
\]
Since $a_-<1$, we must have $C_3=0$ to ensure smoothness at $r=0$ when $l\leq 1$ or $y(0)=0$ when $l=2$.
\end{itemize}
\item{Case 3.2: $a_+<1$ {\normalfont or} $a_+>2$.} In this case, $\UT_+$ obeys \eqref{UT5}. We have two subcases.
\begin{itemize}
\item {Subcase 3.2.1: $a_-=1$ {\normalfont or} $a_-=2$.} This subcase is handled analogously to Subcase 3.1.2 above.

\item {Subcase 3.2.2: $a_-<1$ or $a_->2$.} Here $\UT_-$ also obeys \eqref{UT5}. In order to eliminate the singular term $r^{l-4}$ from \eqref{Eqn:UPMlincomb}, we choose $C_4=-\frac{(1-a_-)(2-a_-)}{(1-a_+)(2-a_+)}C_3$. Then \eqref{Eqn:UPMlincomb} becomes
\[
4C_3r^l\left(\frac{\lambda}{(1-a_+)(2-a_+)} r^{-2}+\BigO(\ln(-r^2/2))\right).
\]
Since $\lambda\neq0$, we see $C_3=0$ is necessary to make the coefficient of $r^{2-l}$ zero and eliminate the singularity at $r=0$ for $l\leq 1$ or to guarantee $y(0)=0$ for $l=2$.
\end{itemize}
\end{itemize}

Finally, suppose $b\geq 4$.
\begin{itemize}
\item {Case 4.1: $1\leq a_+ \leq b-1$.} Throughout this case, $\UT_+$ obeys one of \eqref{UT1}-\eqref{UT3} depending on the value of $b-a_+$.
\begin{itemize}
\item{Subcase 4.1.1: $1\leq a_- \leq b-1$.} In this subcase,  $\UT_-$ also obeys one of \eqref{UT1}-\eqref{UT3} depending on the value of $b-a_-$. Since $b-a_+<b-a_-$, we have four cases to consider. If $b-a_+=1$ and $b-a_-=2$, then we take $C_4=-C_3$ in \eqref{Eqn:UPMlincomb} to eliminate the singular term $r^{2-2b+l}$. Under this imposition \eqref{Eqn:UPMlincomb} becomes
\[
-C_3(b-3)!\left(-\frac{1}{2}\right)^{2-b}r^{4-2b+l}.
\]
Since $b\geq 4$ and $l<b$, we again have $4-2b+l<0$, and so we must choose $C_3=0$ to ensure smoothness at $r=0$. 

When $b-a_+=1$ and $b-a_-\geq 3$, we are forced to take $C_4=-(b-a_--1)!C_3$ in \eqref{Eqn:UPMlincomb} to eliminate $r^{2-2b+l}$, and so \eqref{Eqn:UPMlincomb} becomes
\[
-C_3(b-2)!\left(-\frac{1}{2}\right)^{1-b}r^{2-2b+l}\left(-\frac{1}{2}\left(\frac{b-a_--1}{b-2}\right)r^2+\BigO(r^4)\right),
\]
and arguing as before, we must take $C_3=0$ to preserve smoothness at $r=0$.

When $b-a_+=2$ and $b-a_-\geq 3$, $C_3$ and $C_4$ must be chosen to eliminate both $r^{2-2b+l}$ and $r^{4-2b+l}$ from \eqref{Eqn:UPMlincomb}, again to ensure smoothness when $r=0$. In this case, the constants $C_3$ and $C_4$ must satisfy the system
\begin{align*}
(b-2)!C_3+\frac{(b-2)!}{(b-a_--1)!}C_4&=0,\\
(b-3)!C_3+\frac{(b-3)!}{(b-a_--2)!}C_4&=0.
\end{align*}
This system has a nontrivial solution only when $b-a_--1=1$, which is impossible since $b-a_-\geq 3$. 

Finally, when $b-a_\pm \geq 3$, eliminating $r^{2-2b+l}$ and $r^{4-2b+l}$ from \eqref{Eqn:UPMlincomb} leads to the system of equations
\begin{align*}
\frac{(b-2)!}{(b-a_+-1)!}C_3+\frac{(b-2)!}{(b-a_--1)!}C_4&=0,\\
\frac{(b-3)!}{(b-a_+-2)!}C_3+\frac{(b-3)!}{(b-a_--2)!}C_4&=0,
\end{align*}
which has a nontrivial solution precisely when $b-a_+-1=b-a_--1$. This is impossible, since $a_+\neq a_-$.

\item{Subcase 4.1.2: $a_-<1$ or $a_->b-1$.} Throughout this subcase, $\UT_-$ is defined by \eqref{UT4}. If $b-a_+=1$, to ensure smoothness at $r=0$, the constants $C_3$ and $C_4$ of \eqref{Eqn:UPMlincomb} must satisfy
\begin{align*}
(b-2)!C_3+\frac{(b-1)!(b-2)!}{(1-a_-)_{b-1}}C_4=0,\\
\frac{(b-1)!(b-2)!}{(1-a_-)_{b-1}}\frac{b-a_--1}{b-2}C_4=0.
\end{align*}
This system has a nontrivial solution precisely when the coefficient of $C_4$ in the second equation is nonzero, which is impossible.

When $b-a_+=2$, to ensure smoothness of $y(r)$ at $r=0$, the constants $C_3$ and $C_4$ of \eqref{Eqn:UPMlincomb} must satisfy
\begin{align*}
(b-2)!C_3+\frac{(b-1)!(b-2)!}{(1-a_-)_{b-1}}C_4=0,\\
(b-3)!C_3+\frac{(b-1)!(b-2)!}{(1-a_-)_{b-1}}\frac{b-a_--1}{b-2}C_4=0.
\end{align*}
After a bit of algebra, we see that the above system has a nontrivial solution precisely when $b-a_-=2$, which is impossible.

Finally when $b-a_+\geq 3$, we make the same considerations to $y(r)$ at $r=0$ to the linear combination \eqref{Eqn:UPMlincomb} and are led to the system of equations
\begin{align*}
\frac{(b-2)!}{(b-a_+-1)!}C_3+\frac{(b-1)!(b-2)!}{(1-a_-)_{b-1}}C_4&=0,\\
\frac{(b-2)!}{(b-a_+-1)!}\frac{b-a_+-1}{b-2}C_3+\frac{(b-1)!(b-2)!}{(1-a_-)_{b-1}}\frac{b-a_--1}{b-2}C_4&=0,
\end{align*}
which admits a nontrivial solution precisely when $a_+=a_-$, which is again impossible.
\end{itemize}
\item {Case 4.2: $a_+<1$ {\normalfont or} $a_+>b-1$.} In this case, $\UT_+$ obeys \eqref{UT4}. We have two subcases.
\begin{itemize}
\item {Subcase 4.2.1: $1\leq a_-\leq b-1$.} This subcase is handled identically to Subcase $4.1.2$.
\item {Subcase 4.2.2: $a_-<1$ or $a_->b-1$.} Since $\UT_-$ obeys \eqref{UT4}, canceling both $r^{2-2b+l}$ and $r^{4-2b+l}$ from \eqref{Eqn:UPMlincomb} leads to the system of equations
\begin{align*}
\frac{(b-1)!(b-2)!}{(1-a_+)_{b-1}}C_3+\frac{(b-1)!(b-2)!}{(1-a_-)_{b-1}}C_4&=0,\\
\frac{(b-1)!(b-2)!}{(1-a_+)_{b-1}}\frac{b-a_+-1}{b-2}C_3+\frac{(b-1)!(b-2)!}{(1-a_-)_{b-1}}\frac{b-a_--1}{b-2}C_4&=0,
\end{align*}
which has a nontrivial solution precisely when $b-a_+-1=b-a_--1$ which, again, is impossible since $a_+\neq a_-$.
\end{itemize}\end{itemize}

\end{document}